\documentclass[11pt]{amsart}
\usepackage{graphicx,pstool}
\usepackage{amssymb,amsmath,amsfonts,amsthm,hyperref,here,enumerate,xfrac,xcolor,mathtools}
\numberwithin{equation}{section}
\definecolor{MyDarkBlue}{rgb}{0,0.29,0.7}
\hypersetup{
    colorlinks=true,       
    linkcolor=MyDarkBlue,          
    linkbordercolor=MyDarkBlue,
    citecolor=MyDarkBlue,
    citebordercolor=MyDarkBlue,        
    filecolor=MyDarkBlue,      
    urlcolor=MyDarkBlue           
}
\newcommand{\R}{\mathbb R}
\newcommand{\N}{\mathbb N}

\newtheoremstyle{plain}
  {10pt}
  {10pt}
  {\it}
  {0pt}
  {\bf}
  {}
  {\newline}
  {}
\newtheoremstyle{definition}
  {10pt}
  {10pt}
  {}
  {0pt}
  {\bf}
  {}
  {\newline}
  {}
  
\theoremstyle{plain}

\newtheorem{theorem}{Theorem}[section]
\newtheorem{coro}[theorem]{Corollary}
\newtheorem{lemma}[theorem]{Lemma}
\newtheorem{prop}[theorem]{Proposition}

\theoremstyle{definition}

\newtheorem{definition}[theorem]{Definition}
\newtheorem{example}[theorem]{Example}
\newtheorem{remark}[theorem]{Remark}

\begin{document}

\title{The One-Sided Isometric Extension Problem}


\author{Norbert Hungerb\"uhler and Micha Wasem}
\address{Department of Mathematics, ETH Z\"urich, R\"amistrasse 101,
    8092 Z\"urich, Switzerland}
\email{norbert.hungerbuehler@math.ethz.ch}
\email{micha.wasem@math.ethz.ch}


\date{\today
}

\maketitle

\begin{abstract}
Let $\Sigma$ be a codimension one submanifold of an $n$-dimensional Riemannian manifold $M$, $n\geqslant 2$. We give a necessary condition for an isometric immersion of $\Sigma$ into $\R^{q}$ equipped with the standard Euclidean metric, $q\geqslant n+1$, to be locally isometrically $C^1$-extendable to $M$. Even if this condition is not met, ``one-sided'' isometric $C^1$-extensions may exist and turn out to satisfy a $C^0$-dense parametric $h$-principle in the sense of Gromov.\end{abstract}
\section{Introduction}
Let $(M,g)$ be an $n$-dimensional ($n\geqslant 2$) Riemannian manifold. Unless otherwise stated, all manifolds and metrics are assumed to be smooth. An isometric immersion is a map $u:M\to\R^q$ satisfying
\begin{equation}\label{pullbackequation}
g=u^*g_0,
\end{equation}
where $g_0=\langle \cdot,\cdot\rangle$ denotes the Euclidean metric. If $u$ is in addition a homeomorphism onto its image, we call $u$ an isometric embedding. Recall that $u$ is called \emph{short}, provided the equality in \eqref{pullbackequation} is replaced by $>$ in the sense of quadratic forms, i.e.\ if $g-u^*g_0$ is positive definite. Since $g$ is symmetric, the above system consists of $s_n=n(n+1)/2$ equations and $q$ unknowns. Usually, $s_n$ is referred to as the \emph{Janet-dimension}. For an analytic metric $g$, the Janet-Burstin-Cartan Theorem (see \cite{janet,burstin,cartan}) gives the local existence of an analytic solution to \eqref{pullbackequation} in the formally determined case $q=s_n$. In the smooth category, the local existence of a smooth solution follows from the works by Nash, Gromov, Rokhlin and Greene (see \cite{nash,gromov-rokhlin,greene}) provided $q=s_n+n$. In the formally overdetermined case $q=n+2$, Nash proved in \cite{nash2} a surprising local existence result for $C^1$-maps and showed that in this case every short immersion can be uniformly approximated by $C^1$-isometric immersions. His work was improved by Kuiper in \cite{kuiper} to the case $q=n+1$. The Nash-Kuiper theorem guarantees for example the existence of an isometric $C^1$-embedding of the flat torus into $\R^3$ of which a visualization appeared in \cite{borrelli2,borrelli}. A further refinement in the codimension one case has been obtained by Conti, de Lellis and Sz\'ekelyhidi in \cite{delellis}, where the authors prove the same statement for $C^{1,\alpha}$-isometric immersions provided $\alpha < \frac{1}{1+2s_n}$. The existence of weak solutions to \eqref{pullbackequation} if $q=n$ is treated in \cite{spadaro}.\\
\\
There is an accompanying extension problem related to \eqref{pullbackequation}: Let $\Sigma$ be a hypersurface in $(M,g)$ and let $f:\Sigma\to(\R^{q},g_0)$ be a smooth isometric immersion (embedding). When does $f$ admit an extension to an isometric immersion (embedding) $v:U\to\R^{q}$ satisfying
\begin{equation}\label{dissproblem}
\begin{aligned}v^*g_0 & = g\\
v|_\Sigma & = f,
\end{aligned}
\end{equation}
where $U\subset M$ is a neighborhood of a point in $\Sigma$?
This question was first considered by Jacobowitz in the case of high codimension and high regularity in \cite{jacobowitz}. Jacobowitz derived a necessary condition on the second fundamental forms of $\Sigma$ in $M$ and $f(\Sigma)$ in $\R^q$ respectively for isometric $C^2$-extensions to exist and showed that this condition is ``almost'' sufficient to prove local existence in the analytic and smooth categories requiring the same conditions on the dimension ($q=s_n$ and $q=s_n+n$) as in the respective local existence Theorems above.\\
\\
In the present work we will focus on the low regularity and low codimension case. Using a length comparison argument we will show that Jacobowitz' obstruction to local isometric $C^2$-extensions is also an obstruction to local isometric $C^1$-extendability. However, restricting the neighborhood $U$ in \eqref{dissproblem} to one side of $\Sigma$ only, we will prove the existence of \emph{one-sided isometric $C^1$-extensions} under very mild hypotheses on $\Sigma$ and $f$, providing an analogue of the Nash-Kuiper Theorem for isometric extensions. It turns out that the so obtained isometric $C^1$-extensions satisfy a $C^0$-dense parametric $h$-principle in the sense of Gromov.
\subsection{Main Results}
In local coordinates the problem \eqref{dissproblem} can be reformulated as follows: Equip an open ball in $\R^n$ centered at zero with an appropriate metric $g$ and let the isometric immersion $f:B\to\R^{n+1}$ be prescribed on $B$ which is the intersection of the closure of the ball with $\R^{n-1}\times \{0\}$. The intersections of the ball with $\R^{n-1}\times\R_{\geqslant 0}$ and $\R^{n-1}\times\R_{\leqslant 0}$ are then called \emph{one-sided neighborhoods} of $B$. The image of a one-sided neighborhood of $B$ under the inverse of a local chart will be called also a one-sided neighborhood (of a point in $\Sigma$). In order to state our main results, we need the following

\begin{figure}[H]
\begin{center}
\psfragfig[scale=1.1]{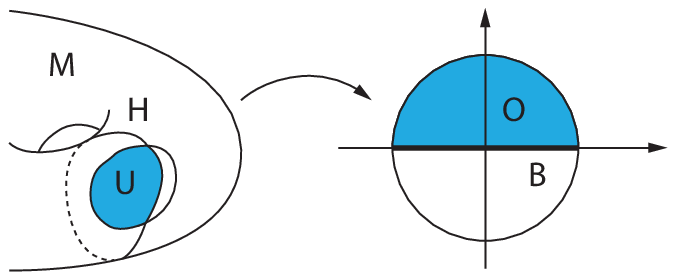}
\end{center}
\end{figure}
\begin{definition}\label{subsolution}
Let $\Omega$ be a one-sided neighborhood of $B$. A $C^\infty$-immersion $u:\bar\Omega\to\R^{n+1}$ is called a \emph{short map adapted to $(f,g)$} whenever $u|_B=f$ and $g-u^*g_0\geqslant 0$ in the sense of quadratic forms with equality on $B$ only. By taking the inverse of a local chart again, we get a notion of an adapted short map on the manifold level.
\end{definition}
We are now ready to state our main results:
\begin{theorem}[$C^1$-Extensions]\label{mainresultc^1}
Let $u:\bar \Omega\to\R^{n+1}$ be a short map adapted to $(f,g)$. Then for every $\varepsilon>0$, there exists a $C^1$-immersion $v:\bar \Omega\to\R^{n+1}$ satisfying $v^*g_0  = g$, $v|_B  = f$ and $\|u-v\|_{C^0(\bar \Omega)}  <\varepsilon$. Moreover, the maps $u$ and $v$ are homotopic within the space of short maps adapted to $(f,g)$, and if $u$ is an embedding we can choose $v$ to be an embedding as well.
\end{theorem}

As a Corollary of Theorem \ref{mainresultc^1} we obtain the following statement about isometric extensions of the standard inclusion $\iota: S^1\hookrightarrow\R^2\times\{0\}\subset \R^3$ to an isometric immersion $S^2\to\R^3$ (here $S^1$ is the equator of $S^2$).
\begin{coro}[Flexible Extensions on $S^2$]\label{flexibleextensions}
There are infinitely many isometric $C^1$-embeddings $v:S^2\to\R^3$ satisfying $v|_{S^1}=\iota$.
\end{coro}
Observe that this Corollary is in sharp contrast to the following uniqueness Theorem (see \cite{borisov1,borisov2,borisov4,borisov3,borisov5,delellis}):

\begin{theorem}[Borisov]\label{uniqueness}
If $\alpha > \frac{2}{3}$, the standard inclusion $S^2\hookrightarrow \R^3$ is the only isometric $C^{1,\alpha}$-extension of $\iota$ to $S^2$ up to reflection across the plane containing $\iota(S^1)$.\end{theorem}

\begin{remark} The threshold $\alpha$ between uniqueness and abundance of solutions is still an open problem (see for instance \cite{laszlo}, \cite{delellis} and \cite[p. 8, Problem 27]{yau}). The proof of Theorem \ref{uniqueness} relies on the conservation of a weak form of Gaussian curvature for $C^{1,\alpha}$-immersions whenever $\alpha > 2/3$ (see \cite{delellis} and the discussion in \cite[p. 369]{fluidequations}).
\end{remark}
\subsection{Organization of the Paper}
In section 2, we will present the obstruction to isometric $C^1$-extendability (Proposition \ref{c1obstruction}) and the construction of adapted short maps (Proposition \ref{constructionofsubsolutions}). The construction of \emph{one-sided isometric $C^1$-extensions} is based on an iteration scheme called \emph{convex integration} which is a generalization by Gromov \cite{gromov} of Nash's original method used in \cite{nash2}. The strategy consists in writing the metric defect of an adapted short map as a sum of \emph{primitive metrics} and add successively error terms. This uses a \emph{Corrugation} that is presented in section 3. In section 4 we give the construction of one-sided isometric $C^1$-extensions. These extensions satisfy a $C^0$-dense parametric $h$-principle. This is the content of section 5. In section 6, we show how one can get an isometric extension that is an embedding. Section 7 indicates how to obtain global results from the local ones by a partition of unity argument. We will formulate all our results for the codimension one case $q=n+1$, but they can be carried on to higher codimension as well (see Remark \ref{codimensionremark}).

\section{Obstructions and Adapted Short Maps}
\subsection{Obstructions}
In this section, we will show that Jacobowitz' necessary condition for the existence of isometric $C^2$-extensions is in fact an obstruction to isometric $C^1$-extensions. Recall that $\Sigma$ is a hypersurface of an $n$-dimensional Riemannian manifold $(M,g)$ and $f:\Sigma\to\R^{n+1}$ is an isometric immersion we seek to extend to a neighborhood $U$ of a point in $\Sigma$.

Let $A\in\Gamma(\mathrm S^2(T^*\Sigma)\otimes N\Sigma)$ denote the second fundamental form of $\Sigma$ in $M$, and$$h(X,Y)\coloneqq g(\nu,A(X,Y)),$$where $\nu\in\Gamma(N\Sigma)$ is a unit vector field. Let further $\bar A\in\Gamma(\mathrm S^2(T^*\Sigma)\otimes f^*N\bar \Sigma)$ be the second fundamental form of $\bar \Sigma\coloneqq f(\Sigma)$ in $\R^q$. In \cite{jacobowitz}, Jacobowitz shows that if $u\in C^2(U,\R^q)$ solves \eqref{dissproblem}, then there exists a unit vector field $\bar \nu\in\Gamma(f^*N\bar \Sigma)$ such that $h(X,Y)=\langle \bar\nu, \bar A(X,Y)\rangle$ for all vector fields $X,Y\in\Gamma(T(\Sigma\cap U))$. In particular, $|h(X,Y)|_g\leqslant |\bar A (X,Y)|$. Hence we get as a corollary:
\begin{coro}[$C^2$-Obstruction]\label{c2obstruction}
If there exists a unit vector $v\in T_p\Sigma$ such that $|h(v,v)|_g>|\bar A(v,v)|$, no isometric extension $u\in C^2(U,\R^q)$ can exist.
\end{coro}
In order to prove the obstruction to isometric $C^1$-extendability, we need the following lemma that seems to appear for the first time in \cite[Theorem 5]{haantjes} but we will give a more modern proof:
\begin{lemma}\label{geodesicdistance}
Let $\gamma:[0,\varepsilon)\to(M,g)$ be a unit speed $C^4$-curve with $\gamma(0)=p$ and let $k_g$ denote the geodesic curvature of $\gamma$ at $p$. Then the geodesic distance from $p$ to $\gamma(t)$ satisfies
\begin{equation*}d(p,\gamma(t))=t-\frac{k_g^2}{24}t^3+O(t^4) \text{ for }t\to 0.\end{equation*}
\end{lemma}
\begin{proof}
In order to fix notation, the Levi-Civita connection of $g$ is denoted by $\nabla$ and we will invoke the Einstein summation convention (summation over repeated indices). Pick geodesic normal coordinates centered at $p$ (see e.g.\ \cite{kobayashi}). Then
$$d(p,\gamma(t))=\int_0^t g\left(\dot\gamma(s),\partial_r(s)\right) \,\mathrm ds,$$
where $$\partial_r(s)\coloneqq (\mathrm d\exp_p)_{\bar \gamma(s)}\left(\bar\gamma(s)\left|\bar \gamma(s)\right|_g^{-1}\right)$$ is the radial vector field and $\bar \gamma(t)\coloneqq \exp^{-1}_p(\gamma(t))$. Using the Gauss-Lemma we find

$$\begin{aligned}d(p,\gamma(t)) &=\int_0^t g\left(\dot\gamma(s),\partial_r(s)\right) \,\mathrm ds\\
& = \int_0^tg\left( (\mathrm d\exp_p)_{\bar \gamma(s)}(\dot{\bar\gamma}(s)),(\mathrm d\exp_p)_{\bar \gamma(s)}\left(\tfrac{\bar\gamma(s)}{\left|\bar \gamma(s)\right|_g}\right)\right)\,\mathrm ds\\
& =  \int_0^tg\left( \dot{\bar\gamma}(s),\tfrac{\bar\gamma(s)}{\left|\bar \gamma(s)\right|_g}\right)\,\mathrm ds\\
& =  \int_0^t\frac{\mathrm d}{\mathrm ds}\left|\bar \gamma(s)\right|_g\,\mathrm ds= |\bar{\gamma}(t)|_g.
\end{aligned}
$$
Consider the expansion
$$\bar\gamma(t)=\dot{\bar{\gamma}}(0)t+\frac{\ddot{\bar{\gamma}}(0)}{2!}t^2+\frac{\dddot{\bar{\gamma}}(0)}{3!}t^3+O(t^4)\text{ for }t\to 0.$$
Fix the coordinates such that $\dot{\gamma}=\dot{\bar\gamma}^i\partial_i$ and such that in zero $\dot{\gamma}(0)=\partial_1|_{t=0}$. Since
$$(\mathrm d\exp_p)_0:T_0(T_pM)\cong T_pM\to T_pM$$
is the identity, we find $\dot{\bar{\gamma}}(0)=\dot\gamma(0)$. Moreover since
$$\left.\nabla_{\dot\gamma}\dot\gamma\right|_{t=0}=\left.\nabla_{\partial_1}\dot\gamma\right|_{t=0}\text{ and }\left.\nabla^2_{\dot\gamma}\dot\gamma\right|_{t=0}=\left.\nabla^2_{\partial_1}\dot\gamma\right|_{t=0},$$ we compute $\nabla_{\partial_1}\dot{\bar\gamma}^i\partial_i=\ddot{\bar\gamma}^k\partial_k+\dot{\bar\gamma}^i\Gamma_{1i}^k\partial_k$ and since the Christoffel symbols vanish at $p$, we obtain
$$\left.\nabla_{\dot{\gamma}}\dot{\gamma}\right|_{t=0}=\ddot{\bar\gamma}(0).$$
For the second covariant derivative we compute
$$\begin{aligned}
\nabla_{\partial_1}^2\dot\gamma & = \nabla_{\partial_1}\left(\ddot{\bar\gamma}^k+\dot{\bar\gamma}^i\Gamma_{1i}^k\right)\partial_k\\
& =\left(\dddot{\bar\gamma}^k+\ddot{\bar\gamma}^i\Gamma_{1i}^k+\dot{\bar\gamma}^i\partial_1\Gamma_{1i}^k\right)\partial_k + \left(\ddot{\bar\gamma}^k+\dot{\bar\gamma}^i\Gamma_{1i}^k\right)\Gamma_{1k}^l\partial_l.
\end{aligned}$$
Evaluating in zero gives
$$
\left.\nabla^2_{\dot{\gamma}}\dot{\gamma}\right|_{t=0}=\dddot{\bar\gamma}(0)+\left.\dot{\bar\gamma}^i\partial_1\Gamma_{1i}^k\partial_k\right|_{t=0}
$$
Using the following identity that holds in normal coordinates (see \cite{liao} equation (6)) 
$$
\partial_l\Gamma_{jk}^i(0)=-\frac{1}{3}\left(R^i_{jkl}(0)+R^i_{kjl}(0)\right)
$$
and $\dot{\gamma}(0)=\partial_1|_{t=0}$, we obtain
$$
\left.\nabla^2_{\dot{\gamma}}\dot{\gamma}\right|_{t=0}=\dddot{\bar\gamma}(0)+\left.\partial_1\Gamma_{11}^k\partial_k\right|_{t=0}=\dddot{\bar{\gamma}}(0)-\frac{1}{3}\left(R^k_{111}(0)+R^k_{111}(0)\right)\partial_k|_{t=0}=\dddot{\bar\gamma}(0),$$
where the last equality follows from the antisymmetry of the curvature tensor in the last two slots. It follows that (for $t\to 0$)
$$
d(p,\gamma(t))^2=|\bar{\gamma}(t)|_g^2 = t^2+\frac{1}{2}\left.g(\nabla_{\dot{\gamma}}\dot\gamma,\dot\gamma)\right|_{t=0}t^3+\frac{1}{4}k_g^2t^4+\frac{1}{3} \left.g(\dot\gamma,\nabla_{\dot{\gamma}}^2\dot{\gamma})\right|_{t=0}t^4+O(t^5)
$$
and since $\gamma$ is parametrized by arc length, $\left.g(\nabla_{\dot{\gamma}}\dot\gamma,\dot\gamma)\right|_{t=0}=0$ and $k_g^2=-\left.g(\dot\gamma,\nabla_{\dot{\gamma}}^2\dot{\gamma})\right|_{t=0}$. This implies
$$
d(p,\gamma(t))^2 = t^2-\frac{k_g^2}{12}t^4+O(t^5).
$$
The desired result follows from applying $\sqrt{1+x}=1+\frac{1}{2}x+O(x^2)$ for $x\to 0$.

\end{proof}
\begin{prop}[$C^1$-Obstruction]\label{c1obstruction}
If there exists a unit vector $v\in T_p\Sigma$ such that $|h(v,v)|_g>|\bar A(v,v)|$, no isometric extension $u\in C^1(U,\R^{q})$ can exist.
\end{prop}
\begin{proof}
We argue by contradiction. Suppose $u$ exists and let $\gamma:[0,\varepsilon)\to \Sigma\cap U$ be a geodesic with $\gamma(0)=p$ and $\dot\gamma(0)=v$ such that $d_M(p,\gamma(t))$ is realized by a minimizing geodesic $\sigma:[0,1]\to U$ for all $t\in[0,\varepsilon)$. Let $\bar p = f(p)$ and $\bar \gamma = f\circ \gamma$. Observe that
$$\begin{aligned}
|u(\sigma(1))-u(\sigma(0))|\leqslant \int_0^1\left|\frac{\mathrm d}{\mathrm dt}(u\circ\sigma)(t)\right|\,\mathrm dt = \int_0^1|\dot\sigma(t)|_g\,\mathrm dt=d_M(p,\gamma(t)),
\end{aligned}$$
hence $d_{\R^{q}}(\bar p,\bar \gamma(t))\leqslant d_M(p,\gamma(t))$. Since we have $k_g(p)  =(\nabla^M_{\dot\gamma}\dot\gamma)(0)=(\nabla^\Sigma_{\dot\gamma}\dot\gamma)(0)+A(v,v)$,
we find$$
d_{M}( p, \gamma(t))=t-\frac{|h(v,v)|_g^2}{24}t^3+O(t^4)\text{ for }t\to 0.
$$
This together with an analogous computation of the geodesic curvature of $\bar\gamma$ in $\bar p$ gives
$$
d_{\R^{q}}(\bar p,\bar \gamma(t))-d_{M}( p, \gamma(t)) =\frac{1}{24}\left(|h(v,v)|_g^2-|\bar A(v,v)|^2\right)t^3+O(t^4) \text{ for }t\to 0
$$
contradicting $d_{\R^{q}}(\bar p,\bar \gamma(t))\leqslant d_M(p,\gamma(t))$.\end{proof}

Observe that the foregoing Proposition does not exclude the existence of a $C^1$-solution to \eqref{dissproblem} on a one-sided neighborhood of $\Sigma$ (the part of $U$ that doesn't contain the geodesic segment $\sigma$ in $M$ that relies $p$ and $\gamma(t)$). However such a ``one-sided'' isometric extension cannot be of class $C^2$ since Corollary \ref{c2obstruction} is a pointwise statement and would apply to points in $\Sigma$.

\subsection{Short Maps} We now present a sufficient condition for the existence of adapted short maps.
\begin{prop}\label{constructionofsubsolutions}
Let $f:\Sigma\to\R^{n+1}$ be an isometric immersion. Suppose there exists a unit normal field $\bar\nu\in\Gamma(f^*N\bar \Sigma)$ such that $h(\cdot,\cdot)-\langle \bar A(\cdot,\cdot),\bar\nu\rangle$ is positive definite. Then around every $p\in \Sigma$, there exists a short map adapted to $(f,g)$. This adapted short map can be chosen to be an embedding.
\end{prop}

\begin{proof}
Choose a submanifold chart around $p\in \Sigma$ as in definition \ref{subsolution} and let $\rho: B\to \Sigma$ be a parametrization. Consider the maps
\begin{eqnarray*}
\psi:B\times[0,\varepsilon]&\to& M \\ (x,t)& \mapsto&  \exp_{\rho(x)}(-t\nu(x))\\
u:B\times[0,\varepsilon]&\to & \R^{n+1}\\ (x,t)&\mapsto & (f\circ\rho)(x)- s(t)\bar\nu(x),\end{eqnarray*}
where $s\in C^\infty(\R_{\geqslant 0},\R)$ satisfies $s(0)=0, s'(0)=1$, $s''(0)<0$ and $\nu\in\Gamma(N\Sigma)$ is the unit vector field such that $\langle A(\cdot,\cdot),\nu\rangle=h(\cdot,\cdot)$. We claim that the map $u$ is a short map adapted to $(f\circ\psi,\psi^*g)$. It is clear from the definition, that $u|_B=f\circ\psi|_B$ and we need to show that $\psi^*g-u^*g_0\geqslant 0$ with equality on $B$ only.
We think of $(M,g)$ as being smoothly and isometrically embedded in some euclidean space (Nash Embedding Theorem \cite{nash}) to consider the Taylor expansion of $\psi$ around $t=0$ (see e.g.\ \cite{monera}):
$$\psi(x,t)=\rho(x)-t\nu(x)+\frac{1}{2}\widetilde A(\nu(x),\nu(x))t^2+O(t^3),$$
where $\widetilde A$ is the second fundamental form of $M$ with respect to that embedding.
A direct computation shows that
$$
\psi^*g-u^*g_0=\begin{pmatrix} \begin{pmatrix}2th_{ij}-2s(t)\langle \bar A(\partial_i,\partial_j),\bar\nu\rangle\end{pmatrix}_{ij} & 0\\
0 &1-s'(t)^2 \end{pmatrix}+O(t^2).
$$
This error is positive definite if and only if $2t\left(h_{ij}-\langle \bar A(\partial_i,\partial_j),\bar\nu\rangle\right)+O(t^2)$ is positive definite, which is the case for small $t>0$ by assumption.
The statement regarding embeddings follows immediately from the compactness of $B\times[0,\varepsilon]$, the fact that $\exp$ is a local diffeomorphism and an appropriate choice of $\varepsilon>0$.
\end{proof}

\begin{example}[Existence of Adapted Short Maps despite the $C^1$-Obstruction]The Euclidean metric of $\R^{n}$ in polar coordinates $(r,\varphi_2,\ldots \varphi_{n})$ reads
$$
\mathrm dr\otimes\mathrm dr + r^2\sum_{i,j=2}^{n}g_{ij}\,\mathrm d\varphi^i\otimes\mathrm d\varphi^j,
$$
where the $g_{ij}$ do not depend on $r$. We now equip $\R^{n}$ with a new metric
$$
\hat g = \mathrm dr\otimes\mathrm dr + \Psi(r)\sum_{i,j=2}^{n}g_{ij}\,\mathrm d\varphi^i\otimes\mathrm d\varphi^j,
$$
where $\Psi$ is a smooth real valued function satisfying $\Psi(1)=1$ and $\Psi'(1)>2$. The map $f:S^{n-1}\hookrightarrow \R^{n}\times\{0\}\subset\R^{n+1}$ is an isometric embedding of $S^{n-1}\subset (\R^{n},\hat g)$ into $(\R^{n+1},g_0)$. 
Let $\partial_i\coloneqq \partial_{\varphi_i}$ for $i\in\{2,\ldots,n\}$ and $h^{\hat g}$ denote the scalar second fundamental form of $S^{n-1}$ in $\R^{n}$. Since $-\partial_r$ is a unit normal vector field on $S^{n-1}$, we find that
$$h_{ij}^{\hat g}=-\mathrm dr(\nabla^{\hat g}_{\partial_i}\partial_j)=-\mathrm dr(\Gamma_{ij}^k\partial_k)=-\Gamma_{ij}^r=\frac{1}{2}\Psi'(1)g_{ij}$$ and hence $h_{ij}^{\hat g}-h_{ij}^{g_{0}}=\frac{1}{2}(\Psi'(1)-2)g_{ij}>0$ in the sense of quadratic forms. Choosing $v\coloneqq \partial_i$ for any $i\in\{2,\ldots,n\}$, the $C^1$-obstruction shows that there exists no $\hat g$-isometric $C^1$-extension of $f$ to a neighborhood of $S^{n-1}$ but according to the previous Proposition, we can construct a short map adapted to $(f,\hat g)$.\end{example}
\section{Convex Integration}
The goal of this section is to turn adapted short maps into one-sided isometric $C^1$-extensions. The construction of these extensions is based on the method of Nash \cite{nash2}, Kuiper \cite{kuiper}, Conti, de Lellis and Sz\'ekelyhidi \cite{delellis}. We start with a short map adapted to $(f,g)$, $u:\bar\Omega\to\R^{n+1}$ and decompose the metric defect into a sum of primitive metrics as
$$
(g-u^*g_0)_x=\sum_{k=1}^ma_k^2(x)\nu_k\otimes\nu_k,
$$
where $a_k^2$ are nonnegative smooth functions on $\bar\Omega\setminus B$ that extend continuously to $B$ and vanish on $B$, $\nu_k\in S^{n-1}$ and $m\in\N$ is a finite number but at most $m_0\leqslant m$ terms in the above sum are non-zero for fixed $x$, where $m_0$ depends only on the dimension $n$. This decomposition is the content of \cite[p. 202, Lemma 1]{laszlo} but we will include a sketch of the proof to convenience the reader.

\begin{lemma}[Decomposition of the Metric Defect into Primitive Metrics]
Let $\mathcal P$ be the space of positive definite $(n\times n)$-matrices. If $A\in\mathcal P$, then there exists a sequence $(\nu_k)_{k\in \N}$ such that $\nu_k\in S^{n-1}$ for all $k$ and a sequence $\mu_k\in C^\infty_c(\mathcal P,[0,\infty))$ such that
$$
A=\sum_k\mu_k^2(A)\nu_k\otimes\nu_k,
$$
where almost all of the coefficients $\mu_k(A)$ are zero.
\end{lemma}
\begin{proof}[Sketch of Proof]
The set $\widehat{\mathcal P}:=\{A\in\mathcal P, \operatorname{tr}A=1\}$ is an open convex subset of the set
$$L:=\{B\in\mathrm{Sym}(n),\operatorname{tr}B=1\},\quad\dim L= \frac{n(n+1)}{2}-1.$$ Therefore, each element of $\widehat{\mathcal P}$ is contained in the interior of a simplex
$$\operatorname{conv}\left(A_1,\ldots,A_{n(n+1)/2}\right)\subset L$$
(by Carath\'eodory's theorem on convex sets). If $A\in \widehat{\mathcal P}$ is contained in a non-degenerate simplex
$$S_j=\operatorname{conv}\left(A^j_1,\ldots,A^j_{n(n+1)/2}\right),$$
where each $A^j_i\in\widehat{\mathcal P}$, one can write $A$ in barycentric coordinates with respect to $S_j$ as
$$
A=\sum_{i=1}^{n(n+1)/2} \mu_{i,j}^2 (A)A^j_i,
$$
where $\mu_{i,j}\in C^\infty(S_j,(0,1))$. Since each $A^j_i$ is diagonalizable one can write
$$
A^j_i=\sum_{k=1}^n(c^j_{i,k})^2(\nu_{i,k}^j)\otimes(\nu_{i,k}^j),
$$
where $c^j_{i,k}\in\R$ and $\nu_{i,k}^j\in S^{n-1}$.
Now take a partition of unity subordinate to a locally finite cover $\mathcal C$ of $\widehat{\mathcal P}$ by non-degenerate open simplices in $L$ in order to obtain the desired (finite) decomposition
$$
A=\sum_k\mu_k^2(A)\nu_k\otimes\nu_k.
$$
For a general element $A\in \mathcal P$, observe that $A=\operatorname{tr}A(\operatorname{tr}A)^{-1}A$ and use the fact that $(\operatorname{tr}A)^{-1}A\in\widehat{\mathcal P}$. By continuity of the functions $\mu_k$, the decomposition applies to the metric defect of an adapted short map too and since $\bar\Omega$ is compact, the decomposition will also be finite (see \cite{laszlo} for more details).
\end{proof}
Once one has such a decomposition of the metric defect, a \emph{stage} consists of $m$ \emph{steps}, of which each aims at adding one primitive metric $a^2(x)\nu\otimes\nu$. Fix orthonormal coordinates in the target so that the metric $u^*g_0$ can be written as $\nabla u^T\nabla u$, where $\nabla u = (\partial_ju^i)_{ij}$. For a specific unit vector $\nu\in S^{n-1}$ and a smooth nonnegative function $a\in C^\infty(\bar\Omega)$, we aim at finding $v:\bar\Omega\to \R^{n+1}$ satisfying $\nabla v^T\nabla v \approx \nabla u^T\nabla u + a^2\nu\otimes \nu$. Nash solved this problem using an ansatz of the form 
\begin{equation}\label{nashtwist}
v(x) = u(x) +\frac{a(x)}{\lambda}\left(\cos (\lambda \langle x,\nu\rangle)\beta_1(x)+\sin (\lambda \langle x,\nu\rangle)\beta_2(x)\right),
\end{equation}
where $\lambda>0$ is a (large) constant and $\beta_i$ are mutually orthogonal unit normal fields, requiring thus 2 codimensions (\emph{Nash Twist}). We will explain this ansatz in more detail below. The improvement to codimension one has first been achieved by Kuiper \cite{kuiper} with the use of a different ansatz (\emph{Strain}). We will use a \emph{Corrugation} introduced by Conti, de Lellis and Sz\'ekelyhidi \cite{delellis} (see equation \eqref{ansatzdelellis}) and modify it slightly in order to achieve the desired metric change within the class of adapted short maps (Fig. \ref{skizze}). First, we give a geometric motivation for the choice of the Corrugation that follows \cite{laszlo}. Choose vectors
$$\widetilde \xi\coloneqq \nabla u \cdot \left(\nabla u^T\nabla u\right)^{-1}\cdot \nu,\qquad \widetilde \zeta\coloneqq \star\left(\partial_1u\wedge\partial_2u\wedge\ldots\wedge\partial_nu\right),$$
where $\star$ denotes the Hodge star with respect to the usual metric and orientation in $\R^{n+1}$. Let
\begin{equation}\label{corrfields}\xi\coloneqq \frac{\widetilde\xi}{|\widetilde\xi|^2},\qquad\zeta\coloneqq\frac{\widetilde\zeta}{|\widetilde\zeta||\widetilde \xi|}.\end{equation}
We use an ansatz of the form
$$
v(x) = u(x)+\frac{1}{\lambda}\left(\Gamma_1(x,\lambda \langle x,\nu\rangle)\xi(x) + \Gamma_2(x,\lambda \langle x,\nu\rangle)\zeta(x)\right),
$$
where $\Gamma\in C^\infty(\bar V\times S^1,\R^2), (x,t)\mapsto \Gamma(x,t)$ is a family of loops still to be constructed. The differential of $v$ reads
$$
\nabla v = \nabla u + \partial_t\Gamma_1\xi\otimes \nu + \partial_t\Gamma_2\zeta\otimes\nu + \frac{1}{\lambda}E,
$$
where $E\coloneqq \xi\nabla_x\Gamma_1+\zeta\nabla_x\Gamma_2+\Gamma_1\nabla\xi + \Gamma_2\nabla\zeta$ and therefore
\begin{equation}\label{pullbackingeneral}
\nabla v^T\nabla v = \nabla u^T\nabla u + \frac{1}{|\widetilde \xi|^2}\left(2\partial_t\Gamma_1+(\partial_t\Gamma_1)^2+ (\partial_t\Gamma_2)^2\right)\nu\otimes\nu + r,
\end{equation}
where
$$
r =\frac{2}{\lambda}\mathrm{Sym}\left(\nabla u^TE\right)+\frac{2}{\lambda}\left(\partial_t\Gamma_1\nu \odot E^T\xi +  \partial_t\Gamma_2\nu\odot E^T\zeta\right) + \frac{1}{\lambda^2}\left(E^TE\right)
$$
and $\mathrm{Sym}(A)\coloneqq \frac{1}{2}(A+A^T)$ denotes the symmetrization of $A$ and $a\odot b \coloneqq \mathrm{Sym}(a\otimes b)$. In order to equate the coefficient of $\nu\otimes\nu$ in \eqref{pullbackingeneral} and $a^2$, $\partial_t\Gamma$ needs to satisfy the circle equation $\left(\partial_t\Gamma_1+1\right)^2+\partial_t\Gamma_2^2 = 1+|\widetilde \xi|^2 a^2$. Since we require $\Gamma$ to be $2\pi$-periodic, we also need $\oint_{S^1}\partial_t\Gamma\,\mathrm dt =0.$ Note that in codimension two, where $\xi$ and $\zeta$ can be replaced by mutually orthogonal normal vectors say $\beta_1$ and $\beta_2$, a similar ansatz leads to the circle equation $\partial_t\Gamma_1^2+\partial_t\Gamma_2^2 =a^2$ which is clearly fulfilled by the choice $\partial_t\Gamma(x,t)\coloneqq a(x)(-\sin(t),\cos(t))$. This explains the \emph{Nash Twist} \eqref{nashtwist}.
\begin{center}\begin{figure}

\psfragfig[scale=0.8]{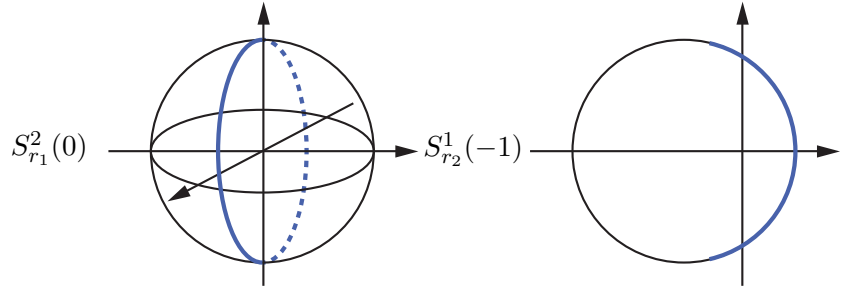}
\caption{The left image illustrates the \emph{Nash Twist} and the right one, the codimension one \emph{Corrugation}.}\label{skizze}

\end{figure}
\end{center}
We look for a $2\pi$-periodic map $\partial_t\Gamma(x,~\cdot~)$ that takes values in a circle parametrized by
$$
(s,t)\mapsto \sqrt{1+s^2}\begin{pmatrix}\cos (f(s)b(t))\\ \sin(f(s)b(t))\end{pmatrix}-\begin{pmatrix}1\\ 0\end{pmatrix},
$$
where $s\coloneqq |\widetilde \xi| a$ and $f(s)$ and $b(t)$ are still to be chosen. In order to satisfy the periodicity condition, we require
$$
\frac{1}{2\pi}\oint_{S^1}\left(\sqrt{1+s^2}\begin{pmatrix}\cos (f(s)b(t))\\ \sin(f(s)b(t))\end{pmatrix}-\begin{pmatrix}1\\ 0\end{pmatrix}\right)\,\mathrm dt = 0.
$$
The second component should be zero independently of $s$ when integrated. This forces $b$ to be $2\pi$-periodic and antisymmetric with respect to $\pi$. The simplest choice is $b(t)\coloneqq \sin t$. For the first component, we aim at finding a function $f$ such that
$$
J_0(f(s))\coloneqq \frac{1}{2\pi}\int_0^{2\pi}\cos(f(s)\sin t)\,\mathrm dt = \frac{1}{\sqrt{1+s^2}}.
$$
Note that $J_0$ is the zeroth Bessel function of the first kind.
\begin{lemma}[Existence of $f$]
There exists a function $f\in C^\infty(\R)$ such that $J_0(f(s))=\frac{1}{\sqrt{1+s^2}}=:w(s)$ satisfying
\begin{equation}\label{festimate}
0<|f'(s)|\leqslant\frac{\sqrt{2+s^2}}{1+s^2}.
\end{equation}
\end{lemma}
\begin{proof}
Observe that $w\in C^\infty(\R)$ takes values in $(0,1]$. Since $J_0:[0,\mu]\to[0,1]$ is a bijection ($\mu$ being the smallest positive zero of $J_0$) and $J'_0$ doesn't admit any zero on $(0,\mu]$, its inverse $J_0^{-1}$ is in $C^{\infty}([0,1))$. Now set
$$
f(s):=\operatorname{sgn}s\cdot J_0^{-1}(w(s))
$$
This function is clearly smooth on $\R\setminus\{0\}$. Since $f$ corresponds around zero to the function constructed by means of the implicit function Theorem in Lemma 2 in \cite{delellis}, $f\in C^\infty(\R)$.
\begin{figure}
\begin{center}
\psfragfig[scale=0.5]{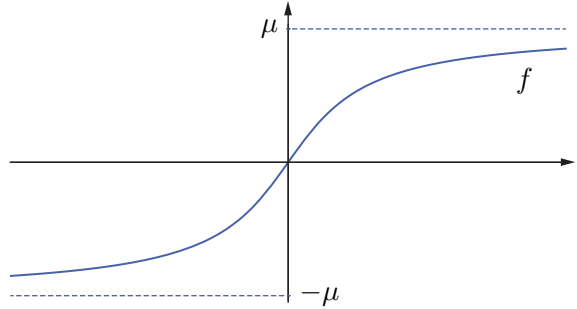}
\caption{The graph of the function $f$ over the interval $[-5,5]$.}
\end{center}
\end{figure}
We will now prove an estimate, from which \eqref{festimate} follows. We will need the following three estimates:
\begin{align}\label{topsoelog}
\frac{2x-2}{x+1}\leqslant\log x & \leqslant \frac{x-1}{\sqrt{x}}, x\geqslant 1\\
\label{boundonJ1}
\frac{|f(s)|}{2\sqrt{1+s^2}} & \leqslant |J_1(f(s))|\\
\label{besselinequality}
\frac{x^2}{8}-\frac{x^4}{96}  \leqslant J_2(x)&\leqslant \frac{x^2}{8}
\end{align}
A proof of \eqref{topsoelog} can be found in \cite{topsoe}. We prove \eqref{boundonJ1} using an integral representation for Bessel functions, trigonometric identities and integration by parts:
\begin{equation}\begin{aligned}\label{besselidentity}
J_1(f(s)) & =\frac{1}{2\pi}\int_0^{2\pi} \sin (f(s)\sin t)\sin t\mathrm dt\\
& = \frac{f(s)}{2\pi}\int_{0}^{2\pi}\cos^2 t\cos(f(s)\sin t)\mathrm dt\\
& = \frac{f(s)}{4\pi}\int_0^{2\pi}\cos(2t)\cos(f(s)\sin t)\mathrm dt + \frac{f(s)}{4\pi}\int_0^{2\pi}\cos(f(s)\sin t)\mathrm dt\\
& = \frac{f(s)}{2}\left(J_2(f(s))+J_0(f(s))\right)\\
& = \frac{f(s)}{2}\left(J_2(f(s))+\frac{1}{\sqrt{1+s^2}}\right).
\end{aligned}
\end{equation}
The identity \eqref{besselidentity} implies \eqref{boundonJ1} since $J_2$ is nonnegative on $[-\mu,\mu]$.
In order to prove \eqref{besselinequality} we use again an integral representation for $J_2$ and integration by parts to obtain
$$
\begin{aligned}
J_2(x) & = \frac{1}{2\pi}\int_0^{2\pi}\cos(2t)\cos(x\sin t)\mathrm dt\\
& =\frac{x}{4\pi}\int_0^{2\pi}\sin(2t)\cos t\sin(x\sin t)\mathrm dt\\
& \leqslant \frac{x^2}{4\pi}\int_0^{2\pi}|\sin (2t)\cos t\sin t|\mathrm dt=\frac{x^2}{8}.
\end{aligned}$$
For the other inequality of \eqref{besselinequality}, we use the same expression for $J_2$ and use $\sin x=x+\mathrm R_3(x)$, where $|\mathrm R_3(x)|\leqslant\frac{|x|^3}{3!}$. It follows that
$$
\begin{aligned}
J_2(x) & = \frac{x}{4\pi}\int_0^{2\pi}\sin(2t)\cos t\left(x\sin t + \mathrm R_3(x\sin t)\right)\mathrm dt\\
&= \frac{x^2}{8}+ \frac{x}{4\pi}\int_0^{2\pi}\sin(2t)\cos t~\mathrm R_3(x\sin t)\mathrm dt
\end{aligned}$$
and hence
$$
\frac{x^2}{8}-J_2(x)  \leqslant \frac{x^4}{24\pi}\int_0^{2\pi}|\sin(2t)\cos t\sin^3 t|\mathrm dt=\frac{x^4}{96}.$$

Using the definition of $f$ implies
$$0<f'(s)=\frac{s}{J_1(f(s))\sqrt{(1+s^2)^3}}$$
and hence \eqref{boundonJ1} implies
$$f'(s) \leqslant \frac{2s}{f(s)(1+s^2)}.$$
Multiplication with $f$ and integration together with the second inequality of \eqref{topsoelog} gives
\begin{equation}\label{boundonf}
|f(s)|\leqslant \sqrt{2\log(1+s^2)}\leqslant \sqrt{\frac{2s^2}{\sqrt{1+s^2}}}.
\end{equation}

We return to the identity \eqref{besselidentity} and use \eqref{besselinequality} together with $\frac{x^2}{16}\leqslant \frac{x^2}{8}-\frac{x^4}{96}$ on $[-\mu,\mu]$ to obtain
\begin{equation}\label{boundonfprime}
\frac{16s}{f(s)(1+s^2)(8+f^2(s)\sqrt{1+s^2})}\leqslant f'(s) \leqslant \frac{32s}{f(s)(1+s^2)(16+f^2(s)\sqrt{1+s^2})}\end{equation}
Multiplication of the first inequality of \eqref{boundonfprime} with $f$ and \eqref{boundonf} implies
$$
\frac{8|s|}{(4+s^2)(1+s^2)}\leqslant \frac{1}{2}\left|(f^2)'(s)\right|
$$
and hence integrating together with the first inequality of \eqref{topsoelog}
$$
|f(s)|\geqslant 2\sqrt{\frac{2}{3}\log\left(\frac{4s^2+4}{s^2+4}\right)}\geqslant4\sqrt{\frac{s^2}{8+5s^2}}.
$$
This lower bound on $f$ can be plugged in the second inequality of \eqref{boundonfprime} to obtain
$$0<|f'(s)|\leqslant \frac{\left(\sqrt{8+5s^2}\right)^3}{2(1+s^2)\left(8+s^2\left(5+\sqrt{1+s^2}\right)\right)}.$$
From this estimate, \eqref{festimate} follows by using $\left(\sqrt{8+5s^2}\right)^3\leqslant \sqrt{8+4s^2}(8+6s^2)$ in the numerator and $5+\sqrt{1+s^2}\geqslant 6$ in the denominator.\end{proof}

With this choice of $f$, let $\Gamma:\R^2\to\R^2$,
$$
\Gamma(s,t)\coloneqq \int_0^t \left(\sqrt{1+s^2}\begin{pmatrix}\cos (f(s)\sin u)\\ \sin(f(s)\sin u)\end{pmatrix}-\begin{pmatrix}1\\ 0\end{pmatrix}\right)\mathrm du.
$$
\begin{lemma}[Corrugation]\label{corrugationfunction}
The function $\Gamma$ is $2\pi$-periodic in the second argument, hence $\Gamma:\R\times S^1\to \R^2$, and it holds that
\begin{equation}
|\partial_t \Gamma(s,t)|\leqslant \sqrt{2}|s|,\label{c1-estimate}
\end{equation}
where the constant $\sqrt{2}$ is optimal.
\end{lemma}
\begin{proof} For the periodicity we compute directly (see also \cite{delellis} for this computation):
$$\begin{aligned}
\Gamma(s,t+2\pi)-\Gamma(s,t) & = \int_t^{t+2\pi} \left(\sqrt{1+s^2}\begin{pmatrix}\cos (f(s)\sin u)\\ \sin(f(s)\sin u)\end{pmatrix}-\begin{pmatrix}1\\ 0\end{pmatrix}\right)\mathrm du\\
& = \int_0^{2\pi} \left(\sqrt{1+s^2}\begin{pmatrix}\cos (f(s)\sin u)\\ \sin(f(s)\sin u)\end{pmatrix}-\begin{pmatrix}1\\ 0\end{pmatrix}\right)\mathrm du\\
& =2\pi\begin{pmatrix}\sqrt{1+s^2}J_0(f(s))-1 \\ 0\end{pmatrix}=\begin{pmatrix}0\\0\end{pmatrix}.
\end{aligned}$$
In order to prove \eqref{c1-estimate}, observe that since $\partial_t\Gamma(0,t)=0$, integrating in $s$ yields
$$\partial_t\Gamma(s,t)=\int_0^s\partial_s\partial_t\Gamma(r,t)\,\mathrm dr,$$
and we need to show that $|\partial_s\partial_t\Gamma|$ is bounded by $\sqrt{2}$. We compute
$$\begin{aligned}
\partial_s\partial_t\Gamma(s,t) =& \frac{s}{\sqrt{1+s^2}}\begin{pmatrix}\cos(f(s)\sin t)\\ \sin(f(s)\sin t)\end{pmatrix}+\\ & +\sqrt{1+s^2}f'(s)\sin t\begin{pmatrix}-\sin(f(s)\sin t)\\ \cos(f(s)\sin t)\end{pmatrix},\end{aligned}$$
hence using \eqref{festimate}
$$
|\partial_s\partial_t\Gamma(s,t)|^2 \leqslant \frac{s^2}{1+s^2}+(1+s^2)[f'(s)]^2\leqslant 2.
$$
Since $\partial_s\partial_t\Gamma(0,\frac{\pi}{2})=\sqrt{2}$, the constant is optimal.\end{proof}

\begin{remark}
Note that Conti, de Lellis and Sz\'ekelyhidi use the same Corrugation function, but in \cite{delellis}, $\Gamma$ is shown to exist on $[0,\varepsilon]\times S^1$ for some small number $\varepsilon>0$ and estimates on the derivatives of all orders are provided. The difference is that here $\Gamma$ is shown to exist on all of $\R\times S^1$, the estimate \eqref{c1-estimate} holds globally and we don't need estimates on the higher order derivatives at the price of getting lower regularity of the solution in the end.
\end{remark}

Adding a primitive metric $a^2\nu\otimes\nu$ to $\nabla u^T\nabla u$ is done with the ansatz
\begin{equation}\label{ansatzdelellis}
v(x) = u(x)+\frac{1}{\lambda}\left(\Gamma_1(a(x)|\widetilde \xi(x)|,\lambda \langle x,\nu\rangle)\xi(x) + \Gamma_2(a(x)|\widetilde \xi(x)|,\lambda \langle x,\nu\rangle)\zeta(x)\right)
\end{equation}
and we say that $v$ is obtained from $u$ by \emph{Convex Integration}.

\begin{remark}\label{codimensionremark}
In higher codimension $q> n+1$, one can always find (locally) a smooth normal vector field (that plays the role of $\zeta$), hence the same construction can be be used in the case of higher codimension. For this reason we will restrict ourselves to the codimension 1 case.
\end{remark}

\begin{example}
If the Corrugation above is applied to the map $u:S^1\to\R^2$ that sends $S^1$ to a circle with radius $\frac{1}{2}$, the above map produces the following picture:
\begin{figure}[H]
\begin{center}
\psfragfig[scale=0.8]{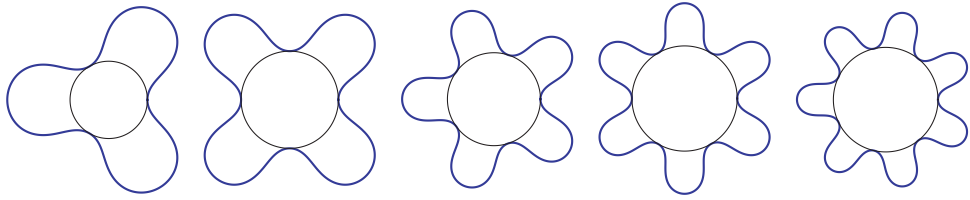}
\caption{$u(S^1)$ and $v(S^1)$ for the choices $\lambda = 3,\ldots,7$}
\end{center}
\end{figure}
\end{example}
\section{Iteration}
\subsection{Step} Since the ansatz \eqref{ansatzdelellis} reaches the desired metric change only up to an error term $O(\lambda^{-1})$, the Corrugation above is not suitable for adapted short maps since it possibly adds a metric defect on $B$, where the adapted short map is already isometric. we will overcome this difficulty by replacing the primitive metric we wish to add by a ``cut-off'' primitive metric that vanishes near $B$. Adding this modified primitive metric can then be done while leaving the initial map $u$ unchanged near $B$. Up to these modifications, we will follow the lines of the Nash-Kuiper iteration scheme as done in the lecture notes by Sz\'ekelyhidi \cite{laszlo}.\\

In order to perform this ``cut-off'', let $\ell>0$ and let $\widetilde\eta_\ell$ be a $C^\infty$-function defined on $\R$ such that $\widetilde\eta_\ell$ vanishes on $\left(-\infty,\ell/2\right]$, is monotonically increasing on $\left[\ell/2,\ell\right]$ and constant with value 1 elsewhere. Let then $\eta_\ell:\bar\Omega\to\R$, $\eta_\ell(x_1,\ldots,x_n)=\widetilde\eta_\ell(x_n)$ and let $\bar\Omega_{j}\coloneqq \{x\in\bar\Omega,\operatorname{dist}(x,B)\leqslant j\}$ where $\operatorname{dist}$ denotes the euclidean distance.

\begin{figure}[H]\begin{center}
\psfragfig[scale=0.9]{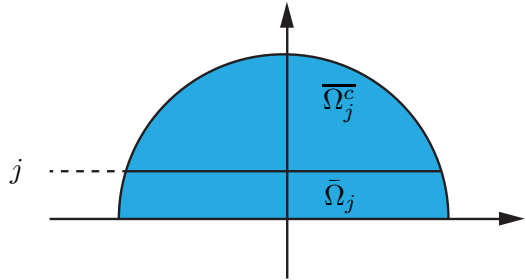}
\caption{Illustration of the definition of $\bar\Omega_j$}
\end{center}\end{figure}
\newpage
\begin{prop}[$k$-th Step]\label{step}
Let $u_{k-1}\in C^\infty(\bar\Omega,\R^{n+1})$ be an immersion. Then for every $\varepsilon>0$ and every $0<\delta<1$ there exists an immersion $u_{k}\in C^{\infty}(\bar\Omega,\R^{n+1})$ depending on a real number $\lambda_k$ that agrees with $u_{k-1}$ on $\bar\Omega_{\ell/2}$ such that the following estimates hold:
\begin{align}
\|u_{k}-u_{k-1}\|_{C^0(\bar\Omega)} & \leqslant \varepsilon,\label{normestimate}\\
\label{diffestimate}
|\nabla u_k-\nabla u_{k-1}|&\leqslant \sqrt 2a_k+O(\lambda_k^{-1}),\\
\label{remestimate}\left\|\nabla u_k^T\nabla u_k -\big[\nabla u_{k-1}^T\nabla u_{k-1} + (1-\delta)\eta_\ell^2a_k^2\nu_k\otimes\nu_k\big]\right\|_{C^0(\bar\Omega)} & \leqslant  \tfrac{\delta^2}{2m}.\end{align}
\end{prop}

\begin{proof}
We use the ansatz
\begin{equation}\label{kstep}
u_k(x)=  u_{k-1}(x)+\frac{1}{\lambda_k}\bigg[\Gamma_1\left(s,\lambda_k \langle x,\nu_k\rangle\right)\xi_k(x)+ \Gamma_2\left(s,\lambda_k \langle x,\nu_k\rangle\right)\zeta_k(x)\bigg],
\end{equation}
where $s\coloneqq (1-\delta)^{1/2}\eta_\ell(x) a_k(x)|\widetilde \xi_k(x)|$ and the vector fields $\xi_k$ and $\zeta_k$ are constructed as in \eqref{corrfields} and $a_k$ is nonnegative. Observe that $a_k$ might fail to be smooth on $B$, but since $\eta_\ell a_k$ is smooth, so is $u_k$.
Since $\Gamma(0,t)=0$ and $\eta_{\ell}|_{\bar\Omega_{\ell/2}}=0$, $u_k$ and $u_{k-1}$ agree on $\bar\Omega_{\ell/2}$.
From \eqref{kstep} we immediately get $|u_k-u_{k-1}|\leqslant \frac{C}{\lambda_k}$, where $C$ depends on $\bar\Omega$ and $k$. Choosing $\lambda_k$ adequately proves \eqref{normestimate}. For the differential one gets
$$\nabla u_k = \nabla u_{k-1}+\partial_t\Gamma_1\xi_k\otimes\nu_k+\partial_t\Gamma_2\zeta_k\otimes\nu_k+\frac{1}{\lambda_k}E_k,$$
where
$$
E_k\coloneqq (1-\delta)^{\sfrac{1}{2}}\partial_s\Gamma_1\xi_k\otimes\operatorname{grad}(\eta_\ell a_k|\widetilde \xi_k|)+(1-\delta)^{\sfrac{1}{2}}\partial_s\Gamma_2\zeta_k\otimes\operatorname{grad}(\eta_\ell a_k|\widetilde \xi_k|) + \Gamma_1\nabla\xi_k+\Gamma_2\nabla\zeta_k.
$$
Using the definition of $\xi_k$ and $\zeta_k$ we obtain
$$
\begin{aligned}
|\partial_t\Gamma_1\xi_k\otimes\nu_k+\partial_t\Gamma_2\zeta_k\otimes\nu_k|^2 & \leqslant  |\left(\partial_t\Gamma_1\xi_k+\partial_t\Gamma_2\zeta_k\right)\otimes\nu_k|^2\\&\leqslant  |\partial_t\Gamma_1\xi_k+\partial_t\Gamma_2\zeta_k|^2\\
& \leqslant |\widetilde \xi_k|^{-2}|\partial_t\Gamma|^2.
\end{aligned}
$$ 
This estimate together with \eqref{c1-estimate} leads to the pointwise estimate
$$\begin{aligned}
|\nabla u_k-\nabla u_{k-1}|&\leqslant |\partial_t\Gamma_1\xi_k\otimes\nu_k+\partial_t\Gamma_2\zeta_k\otimes\nu_k|+\lambda_k^{-1}|E_k|\\
&\leqslant |\widetilde \xi_k|^{-1}|\partial_t\Gamma|+\lambda_k^{-1}|E_k|\\
& \leqslant \sqrt 2(1-\delta)^{\sfrac{1}{2}}a_k\eta_\ell+\lambda_k^{-1}|E_k|\\
&\leqslant \sqrt 2 a_k+\lambda_k^{-1}|E_k|,
\end{aligned}$$
which proves \eqref{diffestimate}. The pullback metric is given by
$$
\nabla u_k^T\nabla u_k = \nabla u_{k-1}^T\nabla u_{k-1} + (1-\delta)\eta_\ell^2a_k^2\nu_k\otimes\nu_k+r_k,
$$
where
$$
r_k\coloneqq \frac{2}{\lambda_k}\mathrm{Sym}\left(\nabla u_{k-1}^TE_k\right)+\frac{2}{\lambda_k}\left(\partial_t\Gamma_1\nu_k\odot E_k^T\xi_k+\partial_t\Gamma_2\nu_k\odot E_k^T\zeta_k\right)+\frac{1}{\lambda_k^2}E_k^TE_k.
$$
Observe that $|E_k|\leqslant C(a_k+\eta_\ell)$, where the constant $C$ depends on $\ell$, $k$ and $\bar\Omega$. This leads to the estimate
\begin{equation}\label{reminder}
|r_k|\leqslant \frac{C}{\lambda_k}(a_k+\eta_\ell + a_k^2+\eta_\ell^2),
\end{equation}
where again, the constant depends on $\ell$, $k$ and $\bar\Omega$. Hence $\|r_k\|_{C^0(\bar\Omega)}\leqslant\frac{\delta^2}{2m}$ provided $\lambda_k$ is large enough. The computation of the pullback metric also implies
$$\nabla u_k^T\nabla u_k \geqslant \nabla u_{k-1}^T\nabla u_{k-1} + r_k,$$
hence $u_k$ is an immersion provided $\lambda_k$ is large enough.
\end{proof}

\subsection{Stage}
A stage consists in adding iteratively ``cut-off'' primitive metrics while controlling the $C^1$-norm of the resulting maps. This is the key ingredient to later obtain the convergence in $C^1(\bar\Omega,\R^{n+1})$. The process leaves the initial map unchanged near $B$.
\begin{prop}[Stage]\label{stage}
Let $u \in C^\infty(\bar\Omega,\R^{n+1})$ be a short map adapted to $(f,g)$. For any $\varepsilon>0$ there exists a map $\widetilde u\in C^\infty(\bar\Omega,\R^{n+1})$ with the following properties:
\begin{align}
\|u-\widetilde u\|_{C^0(\bar\Omega)} & \leqslant   \varepsilon,\label{c0}\\
\|g - \nabla\widetilde u^T\nabla \widetilde u\|_{C^0(\bar\Omega)} & \leqslant   \varepsilon,\label{metricerror}\\
\|\nabla u - \nabla \widetilde u\|_{C^0(\bar\Omega)} & \leqslant  C\|g - \nabla u^T\nabla u\|^{\sfrac{1}{2}}_{C^0(\bar\Omega)}.\label{c1}
\end{align}
Moreover, $\widetilde u$ is an adapted short map with respect to $(f,g)$ provided $\varepsilon>0$ is small enough.
\end{prop}

\begin{proof}
Choose $\ell>0$ such that
\begin{equation}\label{ell}
\|g-\nabla u^T\nabla u\|_{C^0(\bar\Omega_{\ell})}<\frac{\varepsilon}{2}
\end{equation}
and $\delta$ such that the following two conditions are met:
\begin{align}\label{erroroncomplement}\delta\operatorname{id}& <\left(g-\nabla u^T\nabla u\right)|_{\overline{\Omega_{\ell/2}^c}},\\
\label{delta} \delta & <\min\left\{\tfrac{\varepsilon}{2}\|g-\nabla u^T\nabla u\|^{-1}_{C^0(\bar\Omega)},\sqrt{\varepsilon}\right\}.
\end{align}
Now we use Proposition \ref{step} iteratively and choose $\lambda_k$ in each step such that the following three conditions hold:
\begin{align}
\label{conditions1} \|r_k\|_{C^0(\bar\Omega)} & \leqslant   \frac{\delta^2}{2m},\\
\label{conditions2}\|u_k-u_{k-1}\|_{C^0(\bar\Omega)} &  <  \frac{\varepsilon}{m},\\
\label{conditions3}\frac{1}{\lambda_k}\|E_k\|_{C^0(\bar\Omega)} & \leqslant  \frac{1}{m}\|g - \nabla u^T\nabla u\|^{\sfrac{1}{2}}_{C^0(\bar\Omega)}.
\end{align}
We start with the map $u_0=u$ and will get after $m$ steps the desired map $\widetilde u\coloneqq u_m$. We have
$$
\|\widetilde u-u\|_{C^0(\bar\Omega)}\leqslant \sum_{k=1}^m\|u_k-u_{k-1}\|_{C^0(\bar\Omega)}\stackrel{\eqref{conditions2}}{<} \varepsilon.
$$
This proves \eqref{c0}.
We have
\begin{equation}\label{newerror}\begin{aligned}
g-\nabla \widetilde u^T\nabla \widetilde u & =g-\nabla u^T\nabla u+\nabla u^T\nabla u-\nabla \widetilde u^T\nabla \widetilde u\\
& =\sum_{k=1}^m\left(a_k^2\nu_k\otimes\nu_k+\nabla u_{k-1}^T\nabla u_{k-1}-\nabla u_k^T\nabla u_k \right)\\
& =\sum_{k=1}^m\left(a_k^2\nu_k\otimes\nu_k-(1-\delta)\eta_\ell^2a_k^2\nu_k\otimes\nu_k-r_k \right)\\
& =\sum_{k=1}^m\left(\big(1-(1-\delta)\eta_\ell^2\big)a_k^2\nu_k\otimes\nu_k - r_k\right).
\end{aligned}\end{equation}
Using \eqref{ell}, \eqref{delta} and \eqref{conditions1} we get on $\bar\Omega_\ell$:
$$
\|g-\nabla \widetilde u^T\nabla \widetilde u\|_{C^0(\bar\Omega_\ell)} \leqslant \|g-\nabla u^T\nabla u\|_{C^0(\bar\Omega_{\ell})}+\frac{\delta^2}{2}
 < \frac{\varepsilon}{2} + \frac{\varepsilon}{2}\leqslant \varepsilon.
$$
On $\overline{\Omega_\ell^c}$ we use \eqref{delta}, \eqref{conditions1} and \eqref{newerror} to obtain
$$
\|g-\nabla\widetilde u^T\nabla \widetilde u\|_{C^0(\overline{\Omega_\ell^c})}  \leqslant \delta \|g-\nabla u^T\nabla u\|_{C^0(\bar\Omega)} + \frac{\delta^2}{2}< \varepsilon.
$$
which proves \eqref{metricerror}.
For the proof of \eqref{c1}, we use \eqref{diffestimate}, \eqref{conditions3} and the uniform bound
$$
\|g-\nabla u^T\nabla u\|_{C^0(\bar\Omega)} \geqslant  |(g-u^*g_0)_x(\nu_k,\nu_k)|
=\sum_{i=1}^m a_i^2(x)|\langle\nu_k,\nu_i\rangle|^2
 \geqslant a_k^2(x)
$$
to obtain (since at most $m_0$ of the functions $a_k$ are non-zero for fixed $x$) the uniform bound
$$
\begin{aligned}
|\nabla \widetilde u-\nabla u| & \leqslant\sum_{k=1}^m\left(\sqrt 2|a_k|+\lambda^{-1}_{k}\|E_k\|_{C^0(\bar\Omega)}\right)\\
& \leqslant \sqrt 2m_0\|g - \nabla u^T\nabla u\|^{\sfrac{1}{2}}_{C^0(\bar\Omega)}+\|g - \nabla u^T\nabla u\|^{\sfrac{1}{2}}_{C^0(\bar\Omega)}\\
& \leqslant \left(\sqrt 2m_0+1\right)\|g -\nabla u^T\nabla u\|^{\sfrac{1}{2}}_{C^0(\bar\Omega)}.
\end{aligned}
$$
We need to show that the new map $\widetilde u$ is again an adapted short map. Since $\widetilde u|_{\bar\Omega_{\ell/2}}=u_0|_{\bar\Omega_{\ell/2}}$ we need to verify the shortness condition on $\overline{\Omega_{\ell/2}^c}$ only. We use \eqref{erroroncomplement}, \eqref{conditions1} and \eqref{newerror}:
$$
\begin{aligned}
g-\nabla \widetilde u^T\nabla \widetilde u & \geqslant \delta \left(g-\nabla u^T\nabla u\right)-\sum_{k=1}^mr_k\\ & \geqslant \delta^2\operatorname{id} -\sum_{k=1}^mr_k\geqslant \frac{\delta^2}{2}\operatorname{id} .
\end{aligned}
$$
By choosing $\varepsilon>0$ small enough, we ensure that $\widetilde u$ is an immersion. Observe that since $g$ is positive definite on $\bar\Omega$, there exists a positive minimum $\mu$ of the function $\bar\Omega \times S^{n-1}\to \R, (x,v)\mapsto (g)_x(v,v)$ and hence $g \geqslant \mu \operatorname{id}$ in the sense of quadratic forms. Pick any $v\in S^{n-1}$. Then $|(\nabla \widetilde u)_xv|^2  = g_x(v,v)-\left(g-\widetilde u^*g_0\right)_x(v,v)
 \geqslant \mu - \varepsilon>0$, whenever $\varepsilon>0$ is small enough.
\end{proof}

\begin{coro}\label{homotopic steps}
Let $u$, $\widetilde u$ and $\varepsilon>0$ be from the previous Proposition. Then there is a homotopy $\widetilde {\mathrm H}:[0,1]\times\bar\Omega\to\R^{n+1}$ relating $u$ and $\widetilde u$ within the space of short maps adapted to $(f,g)$ and we have the estimates \begin{align}\label{stepclose}
\|\widetilde{\mathrm H} - u\|_{C^0(\bar\Omega)} & \leqslant \varepsilon,\\
\label{homotopsempos}
\widetilde{\mathrm H}(\tau,\cdot)^*g_0-u^*g_0 & \leqslant \left( \|g-u^*g_0\|_{C^0(\bar\Omega)}+\varepsilon\right)\operatorname{id}.
\end{align}

\end{coro}
\begin{proof}
Let $\eta(\tau)\coloneqq \widetilde\eta_{1/2}(\tau)$ and consider the map $\mathrm H: [0,1]\times\bar\Omega\to\R^{n+1}$ given by
$$
(\tau,x)\mapsto u_{k-1}(x)+\frac{1}{\lambda_k}\bigg[\Gamma_1\left(s(\tau),\lambda_k \langle x,\nu_k\rangle\right)\xi_k(x)+ \Gamma_2\left(s(\tau),\lambda_k \langle x,\nu_k\rangle\right)\zeta_k(x)\bigg],
$$
where $s(\tau)\coloneqq \eta(\tau)(1-\delta)^{1/2}\eta_\ell a_k|\widetilde \xi_k|$. Clearly, $\mathrm H$ is a homotopy relating $u_{k-1}$ and $u_{k}$ with the property that $u_{k-1}\equiv\mathrm H(\tau,\cdot)\equiv u_k$ on $\bar\Omega_{\ell/2}$. We must show that whenever $u_{k-1}$ is an adapted short map, so is $\mathrm H(\tau,\cdot)$ for all $\tau$. Observe that $\mathrm H(\tau,\cdot)$ adds the metric term
$$(1-\delta)\eta_{\ell}^2\eta^2(\tau)a_k^2\nu_k\otimes\nu_k+r_k(\tau)$$
to $u_{k-1}^*g_0$. Inequality \eqref{reminder} implies that \eqref{conditions1} holds uniformly in $\tau$ provided $\lambda_k$ is large enough i.e.\ $\|r_k(\tau)\|_{C^0(\bar\Omega)}\leqslant\tfrac{\delta^2}{2m}$.
It follows that
\begin{equation}\label{homotop1}u_{k-1}^*g_0-\frac{\delta^2}{2m}\operatorname{id}\leqslant \mathrm H(\tau,\cdot)^*g_0\leqslant u_{k}^*g_0+\frac{\delta^2}{2m}\operatorname{id}.\end{equation}
We choose $\delta$ in the proof of the previous Proposition such that $u^*g_0>\delta^2\operatorname{id}$ holds additionally. Since
$u_{k}^*g_0-u_{k-1}^*g_0\geqslant r_k$, we find inductively
$$u_k^*g_0> (\delta^2-k\frac{\delta^2}{2m})\operatorname{id}>0$$
and hence
$$u_{k}^*g_0-\frac{\delta^2}{2m}\operatorname{id}>\delta^2\left(\frac{m-1}{2m}\right)\operatorname{id}\geqslant 0,$$
showing that $\mathrm H(\tau,\cdot)$ is never singular.
We are left to show that $g-\mathrm H(\tau,\cdot)^*g_0>0$ on $\overline{\Omega_{\ell/2}^c}$:
$$\begin{aligned}
g-\mathrm H(\tau,\cdot)^*g_0 & \geqslant g-u_k^*g_0-\frac{\delta^2}{2m}\operatorname{id} \\
& \geqslant \sum_{j=1}^k\left((1-(1-\delta)\eta_{\ell}^2)a_j^2\nu_j\otimes\nu_j - r_j\right) + \sum_{j=k+1}^m a_j^2\nu_j\otimes\nu_j-\frac{\delta^2}{2m}\operatorname{id}\\
& \geqslant \sum_{j=1}^m\left(\delta a_j^2\nu_j\otimes\nu_j - r_j\right)-\frac{\delta^2}{2m}\operatorname{id} \geqslant \delta (g - u^*g_0)-\frac{\delta^2(m+1)}{m}\operatorname{id}\\
& >\delta^2\left(\frac{m-1}{2m}\right)\operatorname{id}\geqslant 0.
\end{aligned}$$
The homotopy $\widetilde{\mathrm H}$ relating $u$ and $\widetilde u$ is now obtained from the concatenation of the homotopies relating $u_{k-1}$ and $u_{k}$ for $k=1,\ldots,m$.
From \eqref{conditions2} we know that
$$\|\mathrm H(\tau,\cdot)-u_{k-1}\|_{C^0(\bar\Omega)}\leqslant \frac{\varepsilon}{m}.$$
In particular we get for $\widetilde{\mathrm H}$ the uniform estimate \eqref{stepclose}.
The inequality \eqref{homotop1} implies $u^*g_0-\delta^2\operatorname{id}\leqslant\widetilde{\mathrm H}(\tau,\cdot)^*g_0\leqslant \widetilde u^*g_0+\delta^2\operatorname{id}$ and hence we get using \eqref{newerror} and \eqref{delta}
$$\begin{aligned}
\widetilde{\mathrm H}(\tau,\cdot)^*g_0-u^*g_0 & \leqslant \widetilde u^*g_0-u^*g_0+\delta^2\operatorname{id}\\
& \leqslant \sum_{j=1}^m(1-\delta)\eta_\ell^2a_k^2\nu_k\otimes\nu_k+\delta^2\operatorname{id}\\
& \leqslant \left( \|g-u^*g_0\|_{C^0(\bar\Omega)}+\varepsilon\right)\operatorname{id},
\end{aligned}
$$
thus proving \eqref{homotopsempos}.\end{proof}

\subsection{Passage to the Limit}Proposition \ref{stage} can be used iteratively with an adequate sequence $(\varepsilon_k)_{k\geqslant 1}$ to achieve the convergence in $C^1(\bar\Omega,\R^{n+1})$.

\begin{theorem}[Iteration and $C^0$-density]\label{iteration}
Let $u_0:\bar\Omega\to\R^{n+1}$ be a short map adapted to $(f,g)$. For every $\varepsilon>0$ there exists a sequence $(u_k)_{k\in\N}$ of adapted short maps $u_k\in C^\infty(\bar\Omega,\R^{n+1})$ converging to an isometric immersion $u\in C^1(\bar\Omega,\R^{n+1})$ which coincides with $u_0$ on $B$ and such that $\| u-u_0\|_{C^0(\bar\Omega)}\leqslant \varepsilon$.
\end{theorem}
\begin{proof}
Apply Proposition \ref{stage} iteratively to $u_0$ with a sequence $(\varepsilon_k)_{k\geqslant 1}$ satisfying
$$\sum_{k=1}^\infty\varepsilon_k\leqslant\varepsilon\text{ and }\sum_{k=1}^\infty\sqrt{\varepsilon_k}<\infty$$
and choose each $\varepsilon_k$ such that after the $k$-th stage, the resulting map is again an adapted short map. Since for $j>i$ we have
$$
\begin{aligned}
\|u_j-u_i\|_{C^1(\bar\Omega)} & \leqslant \|u_j-u_i\|_{C^0(\bar\Omega)}+\|\nabla u_j- \nabla u_i\|_{C^0(\bar\Omega)}\\
& \leqslant \sum_{k=i+1}^\infty\|u_k-u_{k-1}\|_{C^0(\bar\Omega)}+\sum_{k=i+1}^\infty\|\nabla u_k- \nabla u_{k-1}\|_{C^0(\bar\Omega)}\\
& \leqslant \sum_{k=i+1}^\infty\varepsilon_k+\sum_{k=i}^\infty\sqrt{\varepsilon_k}\stackrel{i,j\to\infty}{\longrightarrow} 0,
\end{aligned}
$$
$(u_k)_{k\in\N}$ is Cauchy in $C^1(\bar\Omega,\R^{n+1})$ and therefore admits a limit map $u:\bar\Omega\to\R^{n+1}$ satisfying $u^*g_{0}=g$ (since $\varepsilon_k\to0$ as ${k\to\infty}$). This shows that $u$ is immersive. Observe that the limit of the pullback metrics equals the metric pulled back by $u$ since the convergence is in $C^1(\bar\Omega,\R^{n+1})$. Moreover
$$
\|u-u_0\|_{C^0(\bar\Omega)} \leqslant \sum_{k=1}^\infty\|u_k-u_{k-1}\|_{C^0(\bar\Omega)} \leqslant \sum_{k=1}^\infty \varepsilon_k\leqslant\varepsilon.
$$
The equality on $B$ is clear since $u_{k+1}|_B=u_k|_B$ for all $k\in\N$.
\end{proof}
\section{$h$-Principle}
We will now show that there is a homotopy of short maps adapted to $(f,g)$ relating $u$ and $u_0$ from the previous Theorem. This implies that one-sided isometric $C^1$-extensions satisfy an $h$-principle.

\begin{coro}\label{homotopic}
The maps $u$ and $u_0$ from the previous Theorem are homotopic within the space of short maps adapted to $(f,g)$.
\end{coro}
\begin{proof}
Observe that Theorem \ref{iteration} together with Corollary \ref{homotopic steps} delivers a Cauchy sequence $u_k:\bar\Omega\to\R^{n+1}$ in $C^1(\bar\Omega,\R^{n+1})$ and homotopies $h_k$ relating $u_{k-1}$ and $u_k$ within the space of short maps adapted to $(f,g)$. Let $c_k(\tau,x)\coloneqq u_k(x)$ be the constant homotopy and define the following homotopies ($\circledast$ denotes concatenation):
$$\mathrm H_k \coloneqq  h_1 \circledast \left(h_2\circledast\cdots\circledast(h_k\circledast c_k)\right).$$
\begin{figure}[H]
\begin{center}
\psfragfig[scale=0.45]{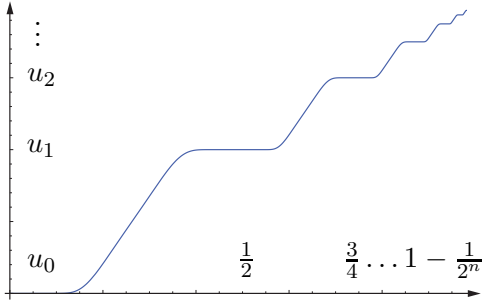}
\caption{Construction of the homotopy $\mathrm H$}
\end{center}
\end{figure}
We will show that
$$
\mathrm H(\tau,x) \coloneqq  \begin{cases} \lim\limits_{k\to\infty}\mathrm H_k(\tau,x), &\text{whenever } \tau\in[0,1)\\ \hfill u(x), &\text{if }\tau =1\end{cases} 
$$
is the desired homotopy between $u$ and $\widetilde u$. We first show that $\mathrm H$ is continuous in $\tau=1$. For fixed $\tau\in[0,1)$, choose $k$ such that $1-\frac{1}{2^{k-1}}\leqslant \tau<1-\frac{1}{2^k}$. Since $k\to\infty$ as $\tau\to 1$, we find
$$
\begin{aligned}
|\mathrm H(\tau,x)-\mathrm H(1,x)| &\leqslant \left|\mathrm H(\tau,x)-\mathrm H\left(\frac{\lfloor 2^k\tau\rfloor}{2^k},x\right)\right|+\left|\mathrm H\left(\frac{\lfloor 2^k\tau\rfloor}{2^k},x\right)- u(x)\right| \\ &\leqslant |\mathrm H(\tau,x)-u_{k-1}(x)|+|u_{k-1}(x)- u(x)|\\ &\leqslant \varepsilon_k+\|u_{k-1}- u\|_{C^0(\bar\Omega)}\stackrel{k\to\infty}{\longrightarrow} 0
\end{aligned}
$$
uniformly in $x$, where the last inequality follows from \eqref{stepclose}. Similarly, we prove continuity of $\mathrm H^*g_0$ in $\tau=1$.
$$
\begin{aligned}
|\mathrm H(\tau,x)^*g_0-g_x| & \leqslant \left|\mathrm H(\tau,x)^*g_0-\mathrm H\left(\frac{\lfloor 2^k\tau\rfloor}{2^k},x\right)^*g_0\right|+\left|\mathrm H\left(\frac{\lfloor 2^k\tau\rfloor}{2^k},x\right)^*g_0-g_x\right| \\ & \leqslant |\mathrm H(\tau,x)^*g_0-(u^*_{k-1}g_0)_x|+|(u^*_{k-1}g_0)_x-g_x|\\ & \leqslant \varepsilon_{k-1}+\varepsilon_k+\|u^*_{k-1}g_0-g\|_{C^0(\bar\Omega)}\stackrel{k\to\infty}{\longrightarrow} 0
\end{aligned}$$
uniformly in $x$, where the last inequality follows from \eqref{homotopsempos}.\end{proof}
Corollary \ref{homotopic} together with Theorem \ref{iteration} means in the language of Gromov \cite{gromov}, Eliashberg and Mishashev \cite{eliashberg}, that the one-sided isometric $C^1$-extensions satisfy a $C^0$-dense $h$-principle. The next Proposition shows that this $C^0$-dense $h$-principle is also parametric, i.e.\ whenever two isometric extensions $u$ are homotopic within the space of adapted short maps, then there is a homotopy of solutions relating them.\\

We start now with a homotopy $\mathrm H:[0,1]\times\bar\Omega\to\R^{n+1}, \mathrm H(\tau,\cdot)\eqqcolon u^\tau$, where $u^0$ and $u^1$ are isometric $C^1$-extensions and $u^\tau$ is a short map adapted to $(f,g)$ for $\tau\in(0,1)$, that is, $u^0$ and $u^1$ are isometric $C^1$-extensions that can be deformed into each other via adapted short maps. The goal is to show that there is a homotopy that carries $u^0$ to $u^1$ in the space of $C^1$-isometric extensions.
\begin{prop}[Parametric Stage]\label{parametricstage}
Let $u^\tau \in C^\infty(\bar\Omega,\R^{n+1})$ be defined as above. For any $\varepsilon>0$ there exists a homotopy $\widetilde{\mathrm H}(\tau,\cdot)\coloneqq \widetilde u^\tau\in C^\infty(\bar\Omega,\R^{n+1})$ such that we have the following estimates uniformly in $\tau$:
\begin{align}
\|u^\tau-\widetilde u^\tau\|_{C^0(\bar\Omega)} & \leqslant   \varepsilon,\label{c0parametric}\\
\|g - (\widetilde u^\tau)^*g_0\|_{C^0(\bar\Omega)} & \leqslant   \varepsilon,\label{metricerrorparametric}\\
\|\nabla u^\tau - \nabla \widetilde u^\tau\|_{C^0(\bar\Omega)} & \leqslant  C\|g - (\nabla u^\tau)^T\nabla u^\tau\|^{\sfrac{1}{2}}_{C^0(\bar\Omega)}.\label{c1parametric}
\end{align}
Moreover $\widetilde u^{0}$ and $\widetilde u^{1}$ are isometric $C^1$-extensions and $\widetilde u^\tau$ is a short map adapted to $(f,g)$ for $\tau\in(0,1)$ provided $\varepsilon>0$ is small enough.\end{prop}
\begin{proof}
We can decompose $g-(u^\tau)^*g_0 = \sum_{k=1}^m (a_k^\tau)^2\nu_k\otimes\nu_k$ and use Proposition \ref{step} with $a_k^\tau$ instead of $a_k$ and replace the function $\eta_\ell$ by a new function $\Theta_\ell(\tau,x)\coloneqq \eta_\ell(x)\eta_\ell(\tau)\eta_\ell(1-\tau)$. This is done because otherwise we cannot have an estimate corresponding to \eqref{erroroncomplement} (see \eqref{erroroncomplementparametric}). With these choices, we can obtain the same estimates as in Proposition \ref{step}: Choose $\ell>0$ such that
\begin{equation}\label{ellparametric}
|g-(\nabla u^\tau)^T\nabla u^\tau|<\frac{\varepsilon}{2}
\end{equation}
holds on $\overline{[0,1]\times\bar\Omega\setminus \Theta_\ell^{-1}(1)}$ and $\delta$ such that
\begin{align}\label{erroroncomplementparametric}\delta\operatorname{id}& < \left(g-(\nabla u^\tau)^T\nabla u^\tau\right)|_{\overline{[0,1]\times\bar\Omega\setminus \Theta_\ell^{-1}(0)}},\\
\label{deltaparametric}
\delta& <\min\left\{\frac{\varepsilon}{2}\left(\max\limits_{\tau\in[0,1]}\|g-(\nabla u^\tau)^T\nabla u^\tau\|_{C^0(\bar\Omega)}\right)^{-1},\sqrt{\varepsilon}\right\}.
\end{align}
We iterate Proposition \ref{step} to get after $m$ steps a homotopy $\widetilde {\mathrm H}(\tau,\cdot)\eqqcolon\widetilde u^\tau$ with $\|\widetilde u^\tau - u^\tau\|<\varepsilon$ uniformly in $\tau$. This shows \eqref{c0parametric} and is done exactly as in the non-parametric case. The computation \eqref{newerror} is replaced by
\begin{equation}\label{newerrorparametric}
g-(\nabla \widetilde u^\tau)^T\nabla\widetilde u^\tau =\sum_{k=1}^m\left(\big(1-(1-\delta)\Theta_\ell^2\big)(a^\tau_k)^2\nu_k\otimes\nu_k - r^\tau_k\right),
\end{equation}
where $\|r^\tau_k\|_{C^0(\bar\Omega)}\leqslant \frac{\delta^2}{2m}$ can be achieved as in Proposition \ref{step}.
Using \eqref{ellparametric} and \eqref{deltaparametric} we get on $\overline{[0,1]\times\bar\Omega\setminus \Theta_\ell^{-1}(1)}$:
$$|g-(\nabla \widetilde u^\tau)^T\nabla\widetilde u^\tau| \leqslant |g-(\nabla u^\tau)^T\nabla u^\tau|+\frac{\delta^2}{2}
 < \frac{\varepsilon}{2} + \frac{\varepsilon}{2}\leqslant \varepsilon.$$
On $\Theta_\ell^{-1}(1)$ we use \eqref{deltaparametric} and \eqref{newerrorparametric} to obtain
$$|g-(\nabla \widetilde u^\tau)^T\nabla\widetilde u^\tau| \leqslant \delta \|g-(\nabla u^\tau)^T\nabla u^\tau\|_{C^0(\bar\Omega)} + \frac{\delta^2}{2}< \varepsilon,$$
which proves \eqref{metricerrorparametric}.
The proof of \eqref{c1parametric} is obtained exactly as in the non-parametric case by choosing $\lambda_k$ in each step large enough:
$$
|\nabla \widetilde u^\tau-\nabla u^\tau|  \leqslant\sum_{k=1}^m\left(C|a^\tau_k|+\frac{1}{\lambda_{k}}\|E^\tau_k\|_{C^0(\bar\Omega)}\right)
 \lesssim \|g - (\nabla u^\tau)^T\nabla u^\tau\|^{\sfrac{1}{2}}_{C^0(\bar\Omega)}.$$
We are left to show that $u^\tau$ is an adapted short map for $\tau\in(0,1)$ and a solution for $\tau=0,1$. Since ${\widetilde u}^\tau\equiv u^\tau$ on $\Theta_\ell^{-1}(0)$ we must only show that ${\widetilde u}^\tau$ is short on $\Theta_\ell^{-1}((0,1])$. We compute
$$
g-(\nabla \widetilde u^\tau)^T\nabla\widetilde u^\tau \geqslant \delta (g-(\nabla  u^\tau)^T\nabla u^\tau)-\sum_{k=1}^mr^\tau_k \geqslant \delta^2\operatorname{id}- \frac{\delta^2}{2}\operatorname{id}>0.
$$
The rest of the proof is exactly the same as in the non-parametric case.
\end{proof}

Theorem \ref{iteration} and its proof can be taken over word by word for the parametric case and delivers the desired homotopy.
\section{From Immersions to Embeddings}
This section follows the lecture notes by Sz\'ekelyhidi \cite{laszlo} very closely. We show that if the adapted short map is an embedding (see Proposition \ref{constructionofsubsolutions}), we can force the isometric extension to be an embedding as well.
\subsection{Step and Stage}
First we show that adding a primitive metric works in the class of embeddings. Taylor's formula implies that $v(y)-v(x)=(\nabla v)_x(y-x)+O(|y-x|^2)$. Write $h\coloneqq x-y$. Then
$$
\begin{aligned}
|v(y)-v(x)|^2 & = \left\langle (\nabla v^T\nabla v)_xh,h\right\rangle + O(|h|^3)\\
& = \left\langle (\nabla u^T\nabla u)_xh,h\right\rangle + (1-\delta)\eta_\ell^2(x)a^2(x)|\langle \nu,h\rangle|^2 +O\left(\lambda^{-1}|h|^2\right) + O(|h|^3)\\
& \geqslant |u(y)-u(x)|^2+O\left(\lambda^{-1}|h|^2\right) + O(|h|^3)\\Ê& = |u(y)-u(x)|^2\left(1+O\left(\lambda^{-1}\right) + O(|h|)\right),
\end{aligned}
$$
hence there exists $\mu>0$ and $\lambda$ large enough, such that $|v(y)-v(x)|\geqslant |u(y)-u(x)|/2$, whenever $|h|<\mu$. The uniform convergence of $\|u-v\|_{C^0(\bar\Omega)}\to0$ as $\lambda\to\infty$ implies that for any $\delta>0$ there exists $\lambda$ large enough such that
$$
\begin{aligned}
|v(y)-v(x)|&\leqslant |v(y)-u(y)|+|u(y)-u(x)|+|u(x)-v(x)|\\
& \leqslant |u(y)-u(x)|+\delta.
\end{aligned}
$$
This in turn implies the uniform convergence
$$
\frac{|v(y)-v(x)|}{|u(y)-u(x)|}\stackrel{\lambda\to\infty}{\to}1
$$
on $\Lambda_\mu\coloneqq \{(x,y)\in\bar\Omega\times\bar\Omega, |y-x|\geqslant\mu\}$. Hence $|v(y)-v(x)|\geqslant |u(y)-u(x)|/2$ whenever $|y-x|\geqslant\mu$ and $\lambda$ large enough. This shows that after a step (and in particular after a stage), we can get an embedding provided $u$ is an embedding.

\subsection{Passage to the Limit}
Now we show that the map $u$ obtained in Theorem \ref{iteration} is an isometric embedding provided $u_0$ is an embedding. Since $\bar\Omega$ is compact and $u$ is immersive, we just need to check, that $u$ is injective. We have
$$
\begin{aligned}
|u(y)-u(x)|^2 & = \left\langle (u^*g_0)_xh,h\right\rangle + o(|h|^2)\\
& = \left\langle(u^*g_0-u_0^*g_0)_xh,h\right\rangle +\left\langle(u_0^*g_0)_xh,h\right\rangle+ o(|h|^2)\\
& = \sum_{k=1}^ma_k^2(x)\left\langle (\nu_k\otimes\nu_k)_xh,h\right\rangle+\left\langle(u_0^*g_0)_xh,h\right\rangle+ o(|h|^2)\\
& \geqslant \left\langle(u_0^*g_0)_xh,h\right\rangle+ o(|h|^2)\\
& = |u_0(x)-u_0(y)|^2\left(1+o(1)\right).
\end{aligned}
$$
As before, we get the existence of a number $\mu>0$ such that
$$|u(x)-u(y)|\geqslant\frac{ |u_0(x)-u_0(y)|}{2}$$
provided $|h|<\mu$. Note that this calculation holds for any $u$ being an isometric extension and is thus independent of the choice of $\varepsilon$ in the Theorem. If $u_0$ is an embedding, we have
$$
\min_{(x,y)\in \Lambda_\mu}|u_0(x)-u_0(y)|\geqslant \widetilde m>0.
$$
Now choose $\varepsilon\coloneqq \frac{\widetilde m}{4}$. The Theorem delivers a map $u$ with $\|u-u_0\|_{C^0(\bar{\Omega})}<\frac{\widetilde m}{4}$, we conclude that
$$
|u_0(x)-u_0(y)|-\frac{\widetilde m}{2} \leqslant |u(x)-u(y)|.
$$
On $\Lambda_\mu$ we thus find $|u(x)-u(y)|\geqslant |u_0(x)-u_0(y)|/2$ which proves that $u$ is an isometric embedding. This together with Theorem \ref{iteration} and Corollary \ref{homotopic} completes the proof of Theorem \ref{mainresultc^1}.

\section{Global $C^1$-Extensions}
In this section, ``global'' has to be understood in the sense that we want to construct solutions to \eqref{dissproblem} on a neighborhood of $\Sigma$ (and not only of a point in $\Sigma$). In order to fix the setting, let $(M,g)$ be an oriented connected and compact Riemannian $n$-manifold and $\Sigma$ an oriented connected and compact codimension one submanifold of $M$ with trivial normal bundle. Let further $f:\Sigma\to\R^{n+1}$ be an isometric immersion.

\subsection{Norms on Manifolds}
Let $\mathcal A\coloneqq \{\varphi_j:U_j\to V_j\subset \R^n\}_j$ be a finite atlas of $M$ such that the $U_j$ are diffeomorphic to open balls in $\R^n$ and such that the maps $\varphi_j$ extend to diffeomorphisms $\varphi_j:\bar U_j\to\bar V_j$. The inverse maps are denoted by $\psi_j$. Let $\mathcal A$ furthermore be the completion of a submanifold atlas of $\Sigma$ in the sense that a coordinate neighborhood $\bar U$ belongs either to the submanifold atlas (i.e.\ maps $\bar U\cap \Sigma$ to $\R^n\times\{0\}$) or doesn't intersect $\Sigma$. 
\begin{definition}\label{norms}
For $f\in C^1(M)$ and $g\in\Gamma(\mathrm S^2(T^*M))$ we define the following norms:
$$\begin{aligned}
\|f\|_{C^0(M)} & =\max_{p\in M}|f(p)|=\max_{j}\|\psi_j^*f\|_{C^0(\bar V_j)},\\
\|\mathrm df\|_{C^0(M)} & \coloneqq \max_{j}\|\psi_j^*\mathrm df\|_{C^0(\bar V_j)}=\max_{j}\|\nabla (f\circ \psi_j)\|_{C^0(\bar V_j)},\\
\|g\|_{C^0(M)} & \coloneqq \max_{j}\|\psi_j^*g\|_{C^0(\bar V_j)},\\
\|f\|_{C^1(M)} & \coloneqq \|f\|_{C^0(M)}+\|\mathrm df\|_{C^0(M)}.
\end{aligned}$$
\end{definition}
Using these definitions we obtain
\begin{lemma}\label{localnorm}
Let $f\in C^0(M)$, $\omega\in\Gamma(T^*M)$ and $g\in \Gamma(\mathrm S^2(T^*M))$ be compactly supported in $\bar U_k$. Then there exists a constant $C$ only depending on the atlas such that the following estimates hold
\begin{align*}
\|f\|_{C^0(M)}&=\|f\circ\psi_k\|_{C^0(\bar V_k)},\\
\|\omega\|_{C^0(M)}&\leqslant  C\|\psi_k^*\omega\|_{C^0(\bar V_k)},\\
\|g\|_{C^0(M)}&\leqslant  C\|\psi_k^*g\|_{C^0(\bar V_k)}.
\end{align*}
\end{lemma}
\begin{proof}
Unwinding the definitions gives the first equality. For the second estimate, one finds
$$
\begin{aligned}
\|\omega\|_{C^0(M)} & = \max_{j}\max_{x\in \bar V_j}\sup_{v\in S^{n-1}}|\omega((\nabla \psi_j)_x(v))|\\
& = \max_{j}\max_{x\in \bar V_j}\sup_{v\in S^{n-1}}|\psi_k^*\omega(\nabla(\varphi_k\circ \psi_j)_x(v))|\\
& = \max_{j}\max_{x\in \bar V_j}\sup_{v\in S^{n-1}}|\nabla(\varphi_k\circ \psi_j)_x(v)|\left|\psi_k^*\omega\left(\tfrac{\nabla(\varphi_k\circ \psi_j)_x(v)}{|\nabla(\varphi_k\circ \psi_j)_x(v)|}\right)\right|\\
& \leqslant \max_{k,j}\|\nabla(\varphi_k\circ \psi_j)\|_{C^0(\bar V_j)}\left\|\psi_k^*\omega\right\|_{C^0(\bar V_k)}\eqqcolon C\left\|\psi_k^*\omega\right\|_{C^0(\bar V_k)}\end{aligned}
$$
and the last estimate is obtained similarly.
\end{proof}
\subsection{One-sided Neighborhoods}
We need to introduce an equivalence relation on $M$ to get an adequate notion of one-sided neighborhood: Pick a submanifold chart $\varphi:\bar U\to \bar V$. The points $p,q\in\bar U$ are equivalent whenever the $n$-th coordinate of $\varphi(p)$ and $\varphi(q)$ have the same sign. Points of coordinate neighborhoods that don't intersect $\Sigma$ are equivalent. The transitive closure of this equivalence relation is denoted by $\sim$ and defines then an equivalence relation on $M$ that clearly doesn't depend on the choice of $\mathcal A$.
\begin{lemma}
It holds that $2\leqslant \# \left(M/_\sim\right)\leqslant 3$.
\end{lemma}
\begin{proof}
The points in $\Sigma$ are not equivalent to the points in $M\setminus \Sigma$. This proves the first inequality. On the other hand each point (path-connectedness) is equivalent to a point in a submanifold chart. In a submanifold chart, $\sim$ has trivially 3 equivalence classes, since the $n$-th component of the coordinate expression of a point is positive, negative or zero.
\end{proof}
\begin{definition}
Let $\bar U$ be a closed neighborhood of $\Sigma$ in $M$. If $\sim$ has three equivalence classes, $\bar U$ can be divided into three parts $U_+, \bar U\cap \Sigma$ and $U_-$ according to $\sim$. We call $\bar U_\pm$ one-sided neighborhoods of $\Sigma$.

If $\sim$ has only two equivalence classes, we restrict our considerations to neighborhoods that split into three equivalence classes and use the definition above. This is always possible since the normal bundle of $\Sigma$ is trivial and diffeomorphic to a tubular neighborhood of $\Sigma$ in $M$ via $\Psi:\Sigma\times(-\varepsilon,\varepsilon)\to M$.
\end{definition}

\subsection{Cut-off}
Let $\Phi\coloneqq \Psi^{-1}$, $\pi:\Sigma\times(-\varepsilon,\varepsilon)\to(-\varepsilon,\varepsilon)$ be the projection onto the second factor and let $\bar \Omega$ be a one-sided neighborhood of $\Sigma$ that has a nonempty intersection with $\Psi(\Sigma\times[0,\varepsilon))$ and set $\eta_\ell:\bar\Omega\to [0,1]$
$$
p\mapsto\begin{cases}(\widetilde\eta_\ell\circ\pi\circ\Phi)(p),&\text{whenever }p\in \Psi(\Sigma\times[0,\varepsilon))\\ \hfill 1 &\text{elsewhere}\end{cases}
$$
and fix the following notation: $\bar\Omega_{\beta}\coloneqq \bar \Omega\cap\Psi(\Sigma\times[0,\beta])$. Observe that $\eta_\ell$ vanishes on $\bar \Omega_{\ell/2}$ and equals one on
$\overline{\Omega_\ell^c}$.

\subsection{Step}
Let $\{\varsigma_j\in C_c^\infty(U_j)\}_j$ be a partition of unity subordinate to $\mathcal A$ in the sense that $\sum\limits_{j}\varsigma^2\equiv 1$.

\begin{definition}
Let $\bar\Omega$ be a one-sided neighborhood of $\Sigma$. A \emph{short map adapted to} $(f,g)$, $u:\bar\Omega\to\R^{n+1}$ is a smooth immersion satisfying $u|_\Sigma=f$ and $g-u^*g_0\geqslant 0$ in the sense of quadratic forms with equality on $\Sigma$ only.
\end{definition}
The metric defect of an adapted short map can be written as $g-u^*g_0=\sum_{j}^{}\varsigma_j^2(g-u^*g_0)$. With the notation $\hat\varsigma_j\coloneqq \varsigma_j\circ\psi_j$, one gets
$$\varsigma_j^2(g-u^*g_0)=\varphi_j^*\left(\hat\varsigma_j^2(\psi_j^*g-(u\circ\psi_j)^*g_0)\right)$$
and with the use of a decomposition into primitive metrics for
$$
\psi_j^*g-(u\circ\psi_j)^*g_0 = \sum_{k=1}^ma^2_{k,j}\nu_k^j\otimes\nu_k^j,
$$
the metric defect can be written as
$$
g-u^*g_0=\sum_{k,j}\varphi_j^*\left((\hat\varsigma_ja_{k,j})^2\nu_k^j\otimes\nu_k^j\right).
$$
Note that $m$ depends on $j$ in general but we will suppress this dependence since we can choose the maximal $m$ over all charts.
\begin{prop}[Global step]\label{globalstep}
Let $u\in C^\infty(\bar \Omega,\R^{n+1})$ be an immersion. For every $\varepsilon>0$ there exists an immersion $v\in C^{\infty}(\bar \Omega,\R^{n+1})$ that agrees with $u$ on $\bar \Omega_{\ell/2}$ such that the following estimates hold:
\begin{align}\label{globalc0error}
\|v-u\|_{C^0(\bar \Omega)}& < \frac{\varepsilon}{\#\mathcal Am},\\ \label{globalc1error}\|\mathrm dv-\mathrm du\|_{C^0(\bar \Omega)} & \leqslant C \|g-u^*g_0\|_{C^0(\bar \Omega)}^{\sfrac{1}{2}},\\
v^*g_0 -\left[u^*g_0 + \varphi_j^*\left((1-\delta)(\hat\varsigma_j a_{k,j}\cdot\eta_\ell\circ\psi_j)^2\nu_k^j\otimes\nu_k^j\right)\right] & \leqslant  \frac{\delta^2}{2\#\mathcal Am}(g-u^*g_0).\label{globalmetricerrorc0}\end{align}
\end{prop}

\begin{proof}
We use the ansatz
$$
v(p)=u(p)+\frac{1}{\lambda_{k,j}}\bigg[\Gamma_1\left(s,\lambda_{k,j} \langle x,\nu^j_{k}\rangle\right)\xi_{k,j}(x)+ \Gamma_2\left(s,\lambda_{k,j} \langle x,\nu^j_k\rangle\right)\zeta_{k,j}(x)\bigg],
$$
where $s\coloneqq (1-\delta)^{1/2}\eta_\ell(\psi_j(x))\hat\varsigma_j(x) a_{k,j}(x)|\widetilde \xi_{k,j}(x)|$, $x=\varphi_j(p)$ and the vector fields $\xi_{k,j}$ and $\zeta_{k,j}$ are constructed as in the local step with the map $u\circ\psi_j$ instead of $u$. Observe that the map $v-u$ is compactly supported in $\bar U_j$. Since in this case
$$\|v-u\|_{C^0(\bar \Omega)}\lesssim\frac{1}{\lambda_{k,j}}\|\Gamma_1\xi_{k,j}+\Gamma_2\zeta_{k,j}\|_{C^0(\bar V_j)},$$
we can choose $\lambda_{k,j}$ large enough to get \eqref{globalc0error}.
For the pullback one finds
$$
\begin{aligned}
v^*g_0 & = \varphi_j^*(v\circ\psi_j)^*g_0\\
& = \varphi_j^*\left((u\circ\psi_j)^*g_0+(1-\delta)(\hat\varsigma_j a_{k,j}\cdot\eta_\ell\circ\psi_j)^2 \nu_k^j\otimes\nu_k^j+r_{k,j}\right)\\
& = u^*g_0 + \varphi_j^*\left((1-\delta)(\hat\varsigma_j a_{k,j}\cdot\eta_\ell\circ\psi_j)^2\nu_k^j\otimes\nu_k^j\right) + \varphi_j^*r_{k,j}.
\end{aligned}
$$
In view of the local step and since $\|\varphi_j^*r_{k,j}\|_{C^0(\bar \Omega)}\lesssim \|r_{k,j}\|_{C^0(\bar V_j)}$ by Lemma \ref{localnorm} we can choose $\lambda_{k,j}$ large enough such that
\begin{equation}\label{restestimate}
\varphi_j^*r_{k,j}|_{\overline{\Omega_{\ell/2}^c}}\leqslant \frac{\delta}{2\#\mathcal A m}(g-u^*g)|_{\overline{\Omega_{\ell/2}^c}}
\end{equation}
in the sense of quadratic forms. This proves \eqref{globalmetricerrorc0}. Observe that $\varphi_j^*r_{k,j}|_{\bar\Omega_{\ell/2}}\equiv 0$. Since $\mathrm dv-\mathrm du$ is also compactly supported in $\bar U_j$, one gets
\begin{equation}\label{erroronderivativesglobal}
\begin{aligned}
\|\mathrm dv-\mathrm du\|_{C^0(\bar \Omega)} & \lesssim \|\nabla((v-u)\circ\psi_j)\|_{C^0(\bar V_j)}\\
& \lesssim \|a_{k,j}\hat\varsigma_j\|_{C^0(\bar V_j)} + O(\lambda_{k,j}^{-1})\\
& \lesssim \|g-u_0^*g_0\|_{C^0(\bar \Omega)}^{\sfrac{1}{2}}
\end{aligned}
\end{equation}
as in the local step. The maps $u$ and $v$ agree on $\bar \Omega_{\ell/2}$ by construction and are immersions (similar argument as in the local step).
\end{proof}

\subsection{Stage}
To prove the global stage, we use Proposition \ref{globalstep} $\#\mathcal A m$ times to prove the same estimates as in the local case.

\begin{prop}[Global Stage]
Let $u \in C^\infty(\bar \Omega,\R^{n+1})$ be a short map adapted to $(f,g)$. For any $\varepsilon>0$ there exists a map $\widetilde u\in C^\infty(\bar \Omega,\R^{n+1})$ with the following properties:
\begin{align}
\|u-\widetilde u\|_{C^0(\bar \Omega)} & \leqslant   \varepsilon,\label{cg0}\\
\|g -\widetilde u^*g_0\|_{C^0(\bar \Omega)} & \leqslant   \varepsilon,\label{globalmetricerror}\\
\|\mathrm du - \mathrm d\widetilde u\|_{C^0(\bar \Omega)} & \leqslant  C\|g - u^*g_0\|^{\sfrac{1}{2}}_{C^0(\bar \Omega)}.\label{cg1}
\end{align}
Moreover, $\widetilde u$ is a short map adapted to $(f,g)$ provided $\varepsilon>0$ is small enough.
\end{prop}
\begin{proof}
Choose $\ell>0$ and $\delta>0$ such that
\begin{align}\label{ellglobal}
\|g-u^*g_0\|_{C^0(\bar \Omega_{\ell})} & <\frac{\varepsilon}{2},\\
\label{deltaglobal}
\delta & <\min\left\{\frac{\varepsilon}{2}\|g-u^*g_0\|^{-1}_{C^0(\bar \Omega)},\sqrt{\varepsilon}\right\}.
\end{align}
Using Proposition \ref{globalstep} iteratively we find after $\#\mathcal A m$ steps a new map $\widetilde u\coloneqq u_{\#\mathcal Am}$ such that
$$
\|\widetilde u - u\|_{C^0(\bar \Omega)}\leqslant\varepsilon,
$$
provided the frequencies $\lambda_{k,j}$ in each step are chosen appropriately. Moreover
$$
\begin{aligned}
g-\widetilde u^*g_0 & = g- u^*g_0 - (\widetilde u^*g_0- u^*g_0)\\
& = \sum_{k,j}^{}\left[\varphi_j^*\left((\hat\varsigma_ja_{k,j})^2\nu_k^j\otimes\nu_k^j\right) - \varphi_j^*\left((1-\delta)(\hat\varsigma_ja_{k,j}\cdot \eta_\ell\circ\psi_j)^2\nu_k^j\otimes\nu_k^j\right) - \varphi_j^*r_{k,j}\right]\\
& = \sum_{k,j}^{}\left[\varphi_j^*\left((1-(1-\delta)(\eta_\ell\circ\psi_j)^2)(\hat\varsigma_ja_{k,j})^2\nu_k^j\otimes\nu_k^j\right) - \varphi_j^*r_{k,j}\right].
\end{aligned}
$$
On $\bar \Omega_\ell$ this yields
$$
\begin{aligned}
g-\widetilde u^*g_0 & \leqslant \sum_{k,j}^{}\left[\varphi_j^*\left((\hat\varsigma_ja_{k,j})^2\nu_k^j\otimes\nu_k^j\right) - \varphi_j^*r_{k,j}\right]\\
& \leqslant g-u^*g_0 - \sum_{k,j} \varphi_j^*r_{k,j}
\end{aligned}
$$
and therefore using \eqref{ellglobal} and \eqref{deltaglobal} $
\|g-\widetilde u^*g_0\|_{C^0(\bar \Omega_\ell)}\leqslant \left\|g-u^*g_0\right\|_{C^0(\bar \Omega_\ell)}+\frac{\delta^2}{2}\leqslant \varepsilon$.
On $\overline{\Omega_\ell^c}$ we find
$$
\begin{aligned}
g-\widetilde u^*g_0 & \leqslant \sum_{k,j}^{}\delta\left[\varphi_j^*\left((\hat\varsigma_ja_{k,j})^2\nu_k^j\otimes\nu_k^j\right) - \varphi_j^*r_{k,j}\right]\\
& \leqslant \delta(g-u^*g_0) - \sum_{k,j} \varphi_j^*r_{k,j}
\end{aligned}
$$
and hence $\|g-\widetilde u^*g_0\|_{C^0(\overline{\Omega_\ell^c})}\leqslant \delta \|g-u^*g_0\|_{C^0(\bar \Omega)}+\frac{\delta^2}{2}\leqslant \varepsilon$. This proves \eqref{globalmetricerror}. For the shortness we find on $\overline{\Omega_{\ell/2}^c}$ using \eqref{restestimate}:
$$
\begin{aligned}
g-\widetilde u^*g_0 & \geqslant \sum_{k,j}^{}\delta\left[\varphi_j^*\left((\hat\varsigma_ja_{k,j})^2\nu_k^j\otimes\nu_k^j\right) - \varphi_j^*r_{k,j}\right]\\
& \geqslant \delta(g-u^*g_0) - \sum_{k,j} \varphi_j^*r_{k,j} \geqslant \delta(g-u^*g_0) - \frac{\delta}{2}(g-u^*g_0)\\
& \geqslant \frac{\delta}{2}(g-u^*g_0)>0.\end{aligned}
$$
The proof of \eqref{cg1} then follows from \eqref{erroronderivativesglobal} and an appropriate choice of frequencies $\lambda_{k,j}$ in each step (as in the local case). The rest of the argument is similar to the one in the local case.\end{proof}
The Iteration Theorem \ref{iteration} and its proof can be taken over word by word to the global case.
\subsection{Applications}

We can now prove Corollary \ref{flexibleextensions}:
\begin{proof}
Whenever $\varepsilon>0$ is small enough, the image of the map
$$\Phi_\varepsilon:[-\tfrac{\pi}{2},\tfrac{\pi}{2}]\times[0,2\pi]  \to\R^3,\quad
(\vartheta,\varphi) \mapsto \left(1-\varepsilon \sin^2\vartheta\right)\begin{pmatrix}\cos\vartheta\cos\varphi \\ \cos\vartheta\sin\varphi\\Ê\sin\vartheta \end{pmatrix}$$
is a 2-dimensional submanifold $S$ of $\R^3$. Since $\Phi_\varepsilon$ is singular at the poles, we parametrize $S$ around the poles by $f_\varepsilon: B_{1/2}(0)\to\R^3$
$$f_\varepsilon(x,y) \coloneqq  (1-\varepsilon(1-x^2-y^2))\left(x,y,\sqrt{1-x^2-y^2}\right)^T.$$
One can check that $S$ is diffeomorphic to $S^2$ ($S^2$ corresponds to the choice $\varepsilon=0$) and that with the induced metric of $\R^3$ we have $g_{S^2}-g_S\geqslant 0$ with equality on the equator only. We can apply now Theorem \ref{iteration} to each of the hemispheres $S^2_{\pm}$ and get sequences of maps $(u^\pm_k)_{k\in \N}$. Observe that for each $k$, the two maps give rise to an embedding $u_k: S^2\to \R^3$ since an open neighborhood of the equator remains unchanged after $k$ steps. The $C^1$-convergence of the sequences implies that also the limit maps
$$v^\pm\coloneqq \lim\limits_{k\to\infty}u^\pm_k$$
give rise to an isometric $C^1$-embedding $v:S^2\to\R^3$ which extends the standard inclusion $S^2\supset S^1\hookrightarrow\R^3$. 
\end{proof}

\begin{example}[Dirichlet Problem]
Let $\Omega=D^2$, $g$ be a Riemannian metric on $\Omega$ and let $f:\partial \Omega\to\R^2\times\{0\}\subset\R^3$ be a smooth isometric embedding. One can consider the Dirichlet Problem for maps $u\in C^1(\bar\Omega,\R^3)$ given by
\begin{equation}\label{dirichlet}\begin{cases}
 \nabla u^T\nabla u  = g & \text{in }\Omega\\
 \hfill u  = f & \text{on }\partial \Omega.
\end{cases}\end{equation}
There is a rigidity theorem for solutions to this problem: In order to state it, recall that $(\bar\Omega,g)$ is called a smooth positive disk, if $g$ has positive Gauss curvature $K_g>0$.
\begin{theorem}[Hong, \cite{hong}]
Let $(\bar\Omega,g)$ be a smooth positive disk with positive geodesic curvature along $\partial\Omega$. Then there is a unique (up to rigid motions) smooth isometric embedding $u:\bar\Omega\to\R^3$ such that $u(\partial\Omega)$ is a planar curve.
\end{theorem}
The global variant of Theorem \ref{mainresultc^1} implies that whenever there exists a short map $u_0$ adapted to $(f,g)$, then for any $\varepsilon>0$, there exists a $C^1$-solution $u$ to \eqref{dirichlet} such that $\|u-u_0\|_{C^0(\bar\Omega)}<\varepsilon$. In order to produce adapted short maps, we can use Hong's theorem: Fix a positive disk $(\bar\Omega,g)$ that satisfies the assumptions of Hong's theorem and consider a perturbed metric $\widetilde g$ that is $C^2$-close to $g$ and such that $g-\widetilde g\geqslant 0$ with equality on $\partial\Omega$ only. This can be achieved by setting $\widetilde g = \varphi g$, where $\varphi:\bar\Omega\to(0,1]$ is a smooth function such that $\varphi|_{\partial\Omega}\equiv 1$, $0<\varphi|_{\Omega}<1$ and such that
$\|1-\varphi\|_{C^2(\bar\Omega)}$ is very small. It follows from the estimate 
$$
\|K_g-K_{\widetilde g}\|_{C^0(\bar\Omega)}\leqslant C\|g-\widetilde g\|_{C^2(\bar\Omega)}
$$
that $(\Omega,\widetilde g)$ also satisfies the assumptions from Hong's theorem. The corresponding isometric embedding $\widetilde u:\bar\Omega\to\R^3$ is then a short map adapted to $(f,g)$, where $f=\widetilde u|_{\partial\Omega}$.
\end{example}

\begin{example}[Coin through Paper Hole]
Consider the map $\gamma_a(t)=C_a(\cos t, a\sin t)^T$, where $a>0$ and $C_a$ is chosen, such that
$$
\int_0^{2\pi}|\dot\gamma_a(t)|\,\mathrm dt = 2\pi.
$$
A reparametrization of $\gamma_a$ by arc length delivers an isometric embedding $f_a:S^1\to\R^2\subset\R^3$. Let $\varepsilon>0$ and consider the maps $\alpha: S^1\times [0,\varepsilon]\to \R^2$ and $\beta:S^1\times[0,\varepsilon]\to\R^3$ given by
$$\alpha(r,t) \coloneqq (1+r+r^2)(\cos t, \sin t)^T\text{ and }\beta(r,t)\coloneqq  f_a(t)+re_3,$$
where $\{e_1,e_2,e_3\}$ denotes the standard basis of $\R^3$. These maps have the property that
$$
\alpha^*g_{\R^2}-\beta^*g_{\R^3}=\begin{pmatrix}4r(1+r) & 0 \\ 0 & r (1 + r) (2 + r + r^2)\end{pmatrix}
$$
and hence the map $\beta\circ\alpha^{-1}$ is short map adapted to $(f_a,g_{\mathbb R^2})$. The global variant of Theorem \ref{iteration} delivers an isometric extension of $f_a$ for every $a$ which can be interpreted as follows: A direct computation shows that
$$\lim\limits_{a\to 0}\operatorname{diam}_{\R^2}(f_a(S^1))=\pi,$$
hence one can cut a circle of radius 1 out of a sheet of paper and push an idealized coin of diameter $<\pi$ through the hole when deforming the paper accordingly.
\end{example}



\end{document}